\documentclass[11pt,reqno]{amsart}

\usepackage[margin=1in]{geometry}

\usepackage[pdftex]{graphicx}
\usepackage{epstopdf}

\title[Ill/well-posedness results for the \MG system]{On the ill/well-posedness and nonlinear instability of the magneto-geostrophic equations}
\date{}

\author{Susan Friedlander}
\address{Department of Mathematics,
University of Southern California, 3620 S.~Vermont Ave.,
Los Angeles, CA 90089} \email{\tt susanfri@usc.edu}

\author{Vlad Vicol}
\address{Department of Mathematics,
The University of Chicago, 5734 University Ave.,
Chicago, IL 60637} \email{\tt vicol@math.uchicago.edu}

\usepackage{amsfonts,amsmath,latexsym,amssymb,verbatim,amsbsy,times,color}
\usepackage{amsthm}
\usepackage{pstricks}

\usepackage{hyperref}

\theoremstyle{plain}
\newtheorem{theorem}{Theorem}[section]
\newtheorem{definition}[theorem]{Definition}
\newtheorem{lemma}[theorem]{Lemma}
\newtheorem{proposition}[theorem]{Proposition}

\theoremstyle{definition}
\newtheorem{remark}[theorem]{Remark}

\def\tilde{\widetilde}
\numberwithin{equation}{section}

\renewcommand\hat{\widehat}
\def\ZZ{{\mathbb Z}}

\def\ZZdstar{ {\mathbb Z}^d_{\ast}}

\def\TT{{\mathbb T}}

\def\Th{\Theta}
\def\Ub{\boldsymbol U}

\def\Teps{\Theta^\epsilon}
\def\Ueps{ {\boldsymbol U}^{\epsilon}}

\def\MG{MG\ }

\def\MGk{\text{MG}_{\kappa}}
\def\AMGk{\text{AMG}_{\kappa}}

\def\MGz{\text{MG}_{0}}
\def\AMGz{\text{AMG}_{0}}

\def\kkx{k_{1}}
\def\kky{k_{2}}
\def\kkz{k_{3}}

\def\kk{\boldsymbol k}
\def\jj{\boldsymbol j}
\def\lb{\boldsymbol l}
\def\xx{\boldsymbol x}

\def\uu{\boldsymbol u}

\def\MM{\boldsymbol M}

\begin{document}


\begin{abstract}
We consider an active scalar equation that is motivated by a model for magneto-geostrophic dynamics and the geodynamo. We prove that the non-diffusive equation is ill-posed in the sense of Hadamard in Sobolev spaces. In contrast, the critically diffusive equation is well-posed (cf.~\cite{FriedlanderVicol}). In this  case we give an example of a steady state that is nonlinearly unstable, and hence produces a dynamo effect in the sense of an exponentially growing magnetic field.
\end{abstract}




\maketitle

\section{Introduction}\label{sec:intro}

The two classical examples of active scalar equations arising fluid dynamics are Burgers' equation from compressible fluids, and the two-dimensional vorticity equation from incompressible fluids. The analytical study of these seemingly simple equations has generated a substantial amount of new mathematics over the past centuries. An active scalar equation from geophysical fluids that has recently received considerable attention in the mathematical literature (see for instance the reference list in \cite{CCW}) is the  surface quasi-geostrophic (SQG) equation. Introduced by Constantin, Majda, and Tabak~\cite{ConstantinMajdaTabak} as a two-dimensional toy model for the three-dimensional fluid equations, the SQG equation has been studied in both the inviscid and the viscous versions,  and recently it was shown  that the critically viscous  equation is globally well-posed \cite{CaffarelliVasseur,KiselevNazarovVolberg}. However the possibility of finite time blow-up in the inviscid SQG equation is still open.

In the present paper we address a class of three-dimensional active scalar equations for which the drift velocity is more singular by one derivative than the active scalar. In contrast, the drift velocity for the SQG equation is of the same order of derivatives as the active scalar. Our motivation for addressing such singular drift equations comes from a model presented by Moffatt~\cite{Moffatt} for the geodynamo and magneto-geostrophic turbulence in the Earth's fluid core. The nonlinear effects in this three-dimensional system are incorporated in an evolution equation with singular drift for a scalar buoyancy field. This magneto-geostrophic (MG) equation has certain features in common with the SQG equation, most notably that $L^\infty$ is the critical Lebesgue space with respect to the natural scaling for both critically diffusive equations. Inspired by the proof of global well-posedness for the critically diffusive SQG equation given by Caffarelli and Vasseur~\cite{CaffarelliVasseur}, we recently used De Giorgi techniques to prove global well-posedness for the critically diffusive MG equation~\cite{FriedlanderVicol}. In contrast, in this present paper we show that the non-diffusive MG and SQG equations are distinctly different at the analytical level, not only because of the more singular drift velocity, but also because of the  structure of the operators relating the velocity to the active scalar.

We study a class of nonlinear active scalar equations for the unknown scalar field $\Th(x,t)$, driven by a singular velocity vector field $\Ub(x,t)$, namely
\begin{align}
  &\partial_t \Th + \Ub \cdot \nabla \Th  =S + \kappa \Delta \Th \label{eq:1}\\
  &\nabla \cdot \Ub = 0\label{eq:2}
\end{align}
where $\kappa \geq 0$ is a physical parameter, and $S(x)$ is a $C^\infty$-smooth source term. The velocity  $\Ub$ is divergence-free, and is obtained from $\Th$ via
\begin{align}
U_j =  \partial_i T_{ij}  \Th \label{eq:4},
\end{align}
for all $j \in \{1,\ldots,d\}$, where $\{ T_{ij}\}$ is a $d\times d$ matrix of Calder\'on-Zygmund operators of convolution type, such that $\partial_i \partial_j T_{ij}\varphi = 0$ for any smooth $\varphi$. Here and throughout this paper we use the summation convention on repeated indices. For simplicity, we consider the domain to be $(\xx,t) \in \TT^d \times (0,\infty) = [0,2\pi]^{d}\times(0,\infty)$. Without loss of generality we may assume that $\int_{\TT^d} \Th(\xx,t)\, dx = 0$ for all $t\geq 0$, since the mean of $\Th$ is  conserved by the flow.


The model proposed by Moffatt is derived from the full MHD equations for an incompressible, rotating, density stratified, electrically conducting fluid. After a series of approximations relevant to the Earth's fluid core,  a linear relationship is established between the velocity and magnetic vectors, and the scalar buoyancy $\Th$. The active scalar equation for $\Th(x,t)$ that contains the nonlinear process in Moffatt's model is precisely \eqref{eq:1}, but where the divergence-free velocity $\Ub$ is explicitly obtained from the buoyancy as
\begin{align}
U_{j} = M_{j} \Th,\label{eq:intro:MG:2}
\end{align}
for $j \in \{1,2,3\}$.
Here the $M_{j}$ are Fourier multiplier operators with symbols given explicitly by
\begin{align}
& \hat{M}_1(\kk) = \frac{2\Omega \kky \kkz |\kk|^2 - (\beta^2/\eta) \kkx \kky^{2} \kkz}{4 \Omega^2 \kkz^{2} |\kk|^2 + (\beta^2/\eta)^2 \kky^{4}}\label{eq:M:1}\\
& \hat{M}_2(\kk) = \frac{-2\Omega \kkx \kkz |\kk|^2 - (\beta^2/\eta) \kky^{3}\kkz}{4 \Omega^2 \kkz^{2} |\kk|^2 + (\beta^2/\eta)^2 \kky^{4}}\label{eq:M:2}\\
& \hat{M}_3(\kk) = \frac{(\beta^2/\eta) \kky^{2}(\kkx^{2} + \kky^{2})}{4 \Omega^2 \kkz^{2} |\kk|^2 + (\beta^2/\eta)^2 \kky^{4}}\label{eq:M:3}
\end{align}where the Fourier variable $\kk \in \ZZ^d$ is such that $\kkz\neq 0$. On $\{k_{3} = 0\}$ we let $\hat{M}_{j}(\kk)=0$, since for self-consistency of the model we assume that $\Th$ and $\Ub$ have zero {\em vertical mean}. It can be directly checked that $k_{j} \cdot \hat{M}_j (\kk) = 0$ for all $\kk \in \ZZ^{d}\setminus \{ k_{3}=0\}$, and hence the velocity field $\Ub$ given by \eqref{eq:intro:MG:2} is indeed divergence-free. It is important to note that although the symbols $\hat{M}_{i}$ are bounded in the region of Fourier space where $k_{1} \leq \max \{ k_{2}, k_{3}\}$, this is not the case on the ``curved'' frequency regions where $k_{3} = {\mathcal O}(1)$ and $k_{2} = {\mathcal O}(|k_{1}|^{r})$, with $0< r \leq 1/2$. In such regions the symbols are unbounded, since as $|k_{1}|\rightarrow \infty$ we have
\begin{align}
|\hat{M}_{1}(k_{1},|k_{1}|^{r},1)| \approx |k_{1}|^{r},\ |\hat{M}_{2}(k_{1},|k_{1}|^{r},1)| \approx |k_{1}|,\ |\hat{M}_{3}(k_{1},|k_{1}|^{r},1)| \approx |k_{1}|^{2r}.
\end{align}
In fact, it may be shown that $|\hat{M}(k)| \leq C |k|$, where $C(\beta,\eta,\Omega)>0$ is a fixed constant, and that this bound is sharp. In~\cite{FriedlanderVicol} we make precise the fact that \eqref{eq:1}--\eqref{eq:2}, with velocity given by \eqref{eq:intro:MG:2}--\eqref{eq:M:3}, is an example of the abstract system \eqref{eq:1}--\eqref{eq:4}, by letting $T_{ij}$ be the zero-order fourier multipliers
\begin{align}\label{eq:Tij}
T_{ij} = -\partial_{i} (-\Delta)^{-1} M_{j}.
\end{align}
We refer to the evolution equation \eqref{eq:1}--\eqref{eq:2}, with drift velocity $\Ub$ given by \eqref{eq:intro:MG:2}--\eqref{eq:M:3}, as the {\em magneto-geostrophic} equation (MG). On the other hand, we refer to the general case when $\Ub$ is given by \eqref{eq:4} as the {\em abstract magneto-geostrophic} equation (AMG). We make this distinction to emphasize that some of the theorems stated in this paper do not make use of the explicit structure of the symbol of $M$, whereas other results, including ill-posedness, make explicit use of the specific structure given by \eqref{eq:M:1}--\eqref{eq:M:3}.  In addition, we shall refer to the {\em diffusive} evolution \eqref{eq:1}, i.e.  $\kappa>0$, as $\MGk$ (respectively $\AMGk$), and to the {\em non-diffusive} case $\kappa = 0$ as $\MGz$ (respectively $\AMGz$).

The physical parameters of the geodynamo model are the following: $\Omega$ is the rotation rate of the Earth, $\eta$ is the magnetic diffusivity of the fluid core, and $\beta$ is the strength of the steady, uniform mean part of the magnetic field in the fluid core. The perturbation magnetic field vector ${\boldsymbol b}(x,t)$ is computed from $\Th(x,t)$ via the operator
\begin{align}
b_{j} =  (\beta/\eta) (-\Delta)^{-1} \partial_{2} M_{j} \Th,\qquad \mbox{for all}\ j \in \{1,2,3\}.\label{eq:magnetic}
\end{align}
In the Earth's fluid core the value of the  diffusivity $\kappa$ is very small. Hence it is relevant to address both the diffusive evolution~\eqref{eq:1}, and the non-diffusive version where $\kappa=0$. As we will demonstrate, the diffusive and the non-diffusive systems have contrasting properties: if $\kappa>0$ the equation is globally well-posed and the solutions are $C^{\infty}$ smooth for positive time~\cite{FriedlanderVicol,FriedlanderVicol2}, whereas for $\kappa=0$ we shall prove that the equation is ill-posed in the sense of Hadamard in Sobolev spaces, but locally well-posed is spaces of analytic functions.

We first address the non-diffusive version of \eqref{eq:1}. We prove  that the $\AMGz$ equations are locally well-posed in the space of real-analytic functions. The proof of analytic well-posedness, a Cauchy-Kowalewskaya-type result, is given in terms of the Gevrey-class norms introduced by Foias and Temam~\cite{FoiasTemam} for the Navier-Stokes equations, combined with ideas from \cite{KukavicaVicolPeriodic,LO} for the Euler equations. Moreover, we show that the breakdown of the real-analytic solutions is fully characterized by a Sobolev norm of the solution, and the initial data, providing an effective self-consistency tool for numerical simulations of the $\AMGz$ equations. In addition, we point out  that if the matrix $T_{ij}$ is self-adjoint, the operator $\Th \mapsto \Ub$ is anti-symmetric, and the arguments of \cite{CCCGW} may be used to obtain the local well-posedness of the $\AMGz$ equations in Sobolev spaces. We emphasize that the specific $\MGz$ operator $\MM$ defined via \eqref{eq:M:1}--\eqref{eq:M:3} is {\em not} anti-symmetric.

In the theory of differential equations, it is classical to call a Cauchy problem {\em well-posed}, in the sense of Hadamard, if given any initial data in a functional space $X$, the problem has a unique solution in $L^{\infty}(0,T;X)$, with $T$ depending only on the $X$-norm of the initial data, and moreover the solution map $Y \mapsto L^{\infty}(0,T;X)$ satisfies strong continuity properties, e.g. it is uniformly continuous, Lipschitz, or even $C^{\infty}$ smooth, for a sufficiently nice space $Y \subset X$.  If one of these properties fail, the Cauchy problem is called {\em ill-posed}. Depending on the specific equation, and on the regularity of the space $X$ considered, one or more properties of a well-posed problem may fail, either locally or globally in time. At one end of this spectrum we have the dramatic examples of {\em blow-up} in finite time for solutions (e.g. \cite{EE} for the Prandtl equations),  while at the other end we have the case in which the solution map is continuous with respect to perturbations in the initial data, but {\em not $C^{k}$-smooth} for some $k> 0$, or at least not $C^{\infty}$-smooth (cf.~\cite{Gerard-VaretDormy,Gerard-VaretNguyen,Grenier,GuoNguyen,Renardy}).  In-between cases of ill-posedness may be due to the fact that $X$ is too large, such as {\em non-uniqueness} of solutions (e.g. for $L^{\infty}$ weak solutions of  the $\MGz$ equation~\cite{Shvydkoy}), or {\em norm explosion}, i.e. arbitrary small data leads to an arbitrarily large solution in arbitrary small time, preventing the solution map to be  continuous (cf.~\cite{BourgainPavlovic,CheskidovShvydkoy} for the fluid equations, or \cite{Tao} for dispersive equations). See Tao's book~\cite{Tao} for a more detailed discussion.

In this paper we prove that the solution map associated to the Cauchy problem for the $\MGz$ equation is not Lipschitz continuous with respect to perturbations in the initial data around a specific steady profile $\Th_{0}$, in the topology of a certain Sobolev space $X$. Hence the Cauchy problem is ill-posed in the sense of Hadamard. In order to prove this we first consider the instability of the $\MGz$ equation, when linearized about a particular steady state. We employ techniques from continued fractions in order to construct an unstable eigenvalue for the linearized operator. The use of continued fractions in a fluid stability problem was introduced by Meshalkin and Sinai~\cite{MeshalkinSinai} for the Navier-Stokes equations and later adapted for the Euler equations by Friedlander, Strauss, and Vishik~\cite{FriedlanderStraussVishik}. In contrast with the Navier-Stokes and Euler equations, where the linear operator generates a bounded semigroup, for the non-diffusive \MG equation we  construct eigenvalues  which have arbitrarily large real part.  This shows, in particular, that the linearized $\MGz$ operator does not generate a semigroup over Sobolev spaces. Ill-posedness in Sobolev spaces for the full nonlinear problem, then follows from perturbation arguments, which consist of showing that we cannot have a solution map that is Lipschitz with respect to the initial data. 

We  emphasize that the mechanism giving the ill-posedness of the nonlinear $\MGz$ equations is not merely the order $1$ derivative loss in the map $\Th \mapsto \Ub$. Rather it is a combination of this derivative loss, with the anisotropy of the Fourier symbol of $M$ in \eqref{eq:M:1}--\eqref{eq:M:3}, {\em and} the fact that the symbol of $M$ is {\em even}. We note  that the even nature of the symbol of $M$ plays a central role in the proof of non-uniqueness for $L^{\infty}$-weak solutions to $\MGz$ given by Shvydkoy in~\cite{Shvydkoy}, via methods from convex integration. In contrast, an example of an active scalar equation where the map $\Th \mapsto \Ub$ is unbounded, but given by an {\em odd Fourier multiplier}, is the limiting case of the inviscid modified SQG equation, introduced by Ohkitani~\cite{Ohkitani}. This equation was recently shown in~\cite{CCCGW} to give a locally well-posed problem in Sobolev spaces (see also Remark~\ref{rem:existence:sobolev} below).

Lastly,  we study the critically diffusive system \eqref{eq:1}--\eqref{eq:4}. We consider the evolution equation \eqref{eq:1} linearized about an arbitrary smooth steady state $\Th_{0}$, and prove that if the associated linearized operator has an unstable eigenvalue, then the full nonlinear $\AMGk$ system is Lyapunov unstable with respect to perturbations in the $L^{2}$ norm. The proof uses a bootstrap argument, a variant of which has been recently employed in a number of fluid contexts (see for example~\cite{FriedlanderPavlovicShvydkoy, FriedlanderPavlovicVicol}). Here the main difficulty arises due to the singular nature of $\Ub$ given by \eqref{eq:4}, which makes it challenging to control the growth of the nonlinearity. To overcome this obstacle we give a new {\em a priori} bound on a sub-critical Sobolev norm of $\Th$, which is uniform in time. This bound does not follow directly from our earlier work~\cite{FriedlanderVicol2}, and we give the necessary details in the Appendix~\ref{appendix}. We then turn to address the linear instability issue for $\MGk$. Using a continued fraction construction, analogous to the treatment given in Section~\ref{sec:ill-posedness} for the case $\kappa= 0$, we demonstrate the existence of unstable, but {\em bounded} eigenvalues. Combining this example with the result that linear implies nonlinear instability proven earlier for $\AMGk$, we conclude that there exists a steady state around which the Moffatt model is nonlinearly unstable. Such an example produces strong growth in the magnetic field ${\boldsymbol b}(x,t)$, proportional to $\exp(t/\kappa)$, which is consistent with the dynamo scenario.

\section{The non-diffusive equations} \label{sec:ill-posedness}

In this section we consider the non-diffusive problem. First we address  the $\AMGz$ equation, which we recall below
\begin{align}
  &\partial_t \Th + \Ub \cdot \nabla \Th =0 \label{eq:Ill:1}\\
  & U_{j} =  \partial_{i} T_{ij} \Th,\ \nabla \cdot \Ub = 0\label{eq:Ill:2}.
  \end{align}


\subsection{Local well-posedness in analytic spaces for the $\AMGz$ equation}
For a large class of equations which arise as singular limits in fluid mechanics (e.g. the Prandtl boundary layer equations, the hydrostatic Euler equations), the strength and structure of the nonlinearity prevents a local existence theory in spaces with finite order of smoothness, such as Sobolev spaces. However, if the derivative loss in the nonlinearity $\Ub\cdot\nabla \Th$ is of order  at most one, both in $\Ub$ and in $\nabla \Th$,  it is possible to obtain the local existence and uniqueness of solutions in spaces of real-analytic functions, in the spirit of a Cauchy-Kowalewskaya result (cf.~\cite{CS,Gerard-VaretDormy,KukavicaTemamVicolZiane} and references therein).  In this section we prove that the $\AMGz$ equation falls in the category of singular equations for which we have analytic well-posedness, and we give a condition that prevents the breakdown of real-analyticity of the solution. Our first result in this direction is the following theorem.
\begin{theorem}[\bf Analytic local well-posedness] \label{thm:analytic}
Let $\Theta(\cdot,0)$ be  real-analytic, with radius of analyticity at least $\tau_{0}$, and analytic norm $K_{0}$. Then, there exists $T = T(\tau_{0},K_{0})>0$ and a unique real-analytic solution on $[0,T)$ to the initial value problem associated to \eqref{eq:Ill:1}--\eqref{eq:Ill:2}.
\end{theorem}

Classically, a real-analytic function $f(x)$ is characterized by a bound of the type $|\partial^\alpha f(x)| \leq K |\alpha|! \tau^{-|\alpha|}$, for all $x$ and all multi-indices $\alpha$, where $K,\tau > 0$ are constants. The supremum over all values of $\tau$ for which such a bound is available is the analyticity radius of $f$. In the periodic setting of this paper there exists an alternate characterization of real-analytic functions: $f$ is real-analytic with radius proportional to $\tau>0$ if and only if the Fourier coefficients $\hat{f}(\kk)$ decay like $\exp(-\tau |\kk|)$ as $|\kk| \rightarrow \infty$.

\begin{proof}[Proof of Theorem~\ref{thm:analytic}]
As in \cite{FoiasTemam,KukavicaVicolPeriodic,LO}, we use the Gevrey-class real-analytic norm
\begin{align*}
\Vert \Th \Vert_\tau^2 = \Vert (-\Delta)^{r/2} e^{\tau (-\Delta)^{1/2}} \Th \Vert_{L^{2}}^2 = \sum_{\kk \in \ZZdstar} |\kk|^{2r} e^{2\tau |\kk|} |\hat{\Th}(\kk)|^{2},
\end{align*}
where $\tau>0$ denotes the analyticity radius, $r> d/2 + 3/2$ is the Sobolev exponent, and since we work in the mean-free setting we let $\ZZdstar = \ZZ^{d} \setminus \{0\}$. We have the a priori estimate
\begin{align}
  \frac{1}{2} \frac{d}{dt} \Vert \Th \Vert_{\tau}^2 = \dot{\tau} \Vert (-\Delta)^{1/4} \Th \Vert_{\tau}^2 + \langle \Ub \cdot \nabla \Th, (-\Delta)^r e^{2 \tau (-\Delta)^{1/2}} \Th \rangle = \dot{\tau} \Vert (-\Delta)^{1/4} \Th \Vert_{\tau}^2 + {\mathcal R},\label{eq:BKM:ode}
\end{align}
where $\langle \cdot, \cdot\rangle$ is the standard $L^2$ inner product. Using Plancherel's theorem, the nonlinear term ${\mathcal R}$ may be written as
\begin{align*}
{\mathcal R} = i (2\pi)^{d} \sum_{\jj+\kk=\lb; \jj,\kk,\lb\in \ZZdstar}  \hat{\Ub}(\jj) \cdot \kk\, \hat{\Th}(\kk)  |\lb|^{2r} e^{2 \tau |\lb|} \hat{\Th}(-\lb).
\end{align*}
Since the the matrix $T_{ij}$ consists of Fourier multipliers with bounded symbols, we obtain from \eqref{eq:Ill:2} that $|\hat{\Ub}(\jj)| \leq |\jj| |\hat{\Th}(\jj)|$ for all $\jj \in \ZZdstar$, and hence
\begin{align}
  {\mathcal R}&\leq C \sum_{\jj+\kk=\lb; \jj,\kk,\lb\in \ZZdstar} |\jj| |\kk| (|\jj|^r + |\kk|^r) |\hat{\Th}(\jj)| e^{\tau|\jj|} |\hat{\Th}(\kk)| e^{\tau |\kk|} |\lb|^r |\hat{\Th}(\lb)| e^{\tau |\lb|}\notag\\
  &\leq C \sum_{\jj+\kk=\lb; \jj,\kk,\lb\in \ZZdstar}  (|\jj|^{r+1/2} |\kk|^{3/2} + |\kk|^{r+1/2} |\jj|^{3/2}) |\hat{\Th}(\jj)| e^{\tau|\jj|} |\hat{\Th}(\kk)| e^{\tau |\kk|} |\lb|^{r+1/2} |\hat{\Th}(\lb)| e^{\tau |\lb|} \label{eq:BKM:est:1}
\end{align}
for some positive constant $C = C(r)$. In the above estimate we have used the triangle inequality $|\lb| \leq |\jj| + |\kk|$, and  $|\kk|^{1/2} \leq |\jj|^{1/2} + |\lb|^{1/2} \leq 2 |\jj|^{1/2} |\lb|^{1/2}$ since $|\jj|, |\lb|\geq 1$. Since the right side of estimate \eqref{eq:BKM:est:1} is symmetric with respect to $|\jj|$ and $|\kk|$, we need to only give the bound for the case $|\jj| \leq |\kk|$. Summing first in $\kk$, and using the Cauchy-Schwartz inequality we obtain from \eqref{eq:BKM:est:1} that
\begin{align}
{\mathcal R} &\leq C \Vert (-\Delta)^{1/4} \Th \Vert_{\tau}^{2} \sum_{\jj \in \ZZdstar} |\jj|^{3/2} |\hat{\Th}(\jj)| e^{\tau |\jj|} \leq C_{r} \Vert (-\Delta)^{1/4} \Th \Vert_{\tau}^{2} \Vert \Th \Vert_{\tau} \label{eq:BKM:est:2}
\end{align}
since $r>d/2 + 3/2$, for some positive constant $C_{r}$ depending only on $r$ and universal constants. From \eqref{eq:BKM:ode} and \eqref{eq:BKM:est:2} we obtain the a priori bound
\begin{align}
\frac{1}{2} \frac{d}{dt} \Vert \Th (\cdot,t)\Vert_{\tau(t)}^2 \leq \left( \dot{\tau}(t) + C_r \Vert\Th (\cdot,t)\Vert_{\tau(t)}\right) \Vert  (-\Delta)^{1/4}  \Th (\cdot,t)\Vert_{\tau(t)}^2
\label{eq:BKM:est:3}.
\end{align}
Having arrived at \eqref{eq:BKM:est:3}, the argument is the same as in \cite{KukavicaVicolPeriodic,LO}.  More precisely, we may let $\tau(t)$ be decreasing and satisfy the ordinary differential equation
\begin{align}
\dot{\tau} + 2 C_{r} K_{0} = 0 \label{eq:BKM:tau:def}
\end{align}
with initial condition $\tau(0) = \tau_{0}$, so that the right side of \eqref{eq:BKM:est:3} is negative, and therefore
\begin{align*}
\Vert \Th(\cdot,t) \Vert_{\tau(t)} \leq K_{0} = \Vert \Th(\cdot,0) \Vert_{\tau_{0}}
\end{align*}
as long as $\tau(t) > 0$. These arguments prove the existence of a real-analytic solution $\Th(t)$ on $[0,T_{\ast})$, where  the maximal time of existence of the real-analytic solution is the time it takes the analyticity radius $\tau(t)$ to reach $0$, i.e. $T_{\ast} = \tau_0 / (2 C_r K_{0})$.  This proof may be made formal using a standard Picard iteration argument.  We omit further details.
\end{proof}

Whereas Theorem~\ref{thm:analytic} guarantees the local in time existence and uniqueness of a real-analytic solution to \eqref{eq:Ill:1}--\eqref{eq:Ill:2}, it gives almost no intuition on what happens to the real-analytic solution at the time when the analyticity radius becomes $0$. In Theorem~\ref{thm:BKM} below we prove that if at time $T$ a certain {\em Sobolev} norm of the solution is suitably controlled in terms of the initial data, then the real-analytic solution may be continued past $T$; therefore the  breakdown of real-analyticity is controlled only by the initial data and only a Sobolev norm of the solution.

\begin{theorem}[{\bf Criterion for the breakdown of analyticity}]\label{thm:BKM}
Assume the initial data $\Th(\cdot,0)$ is real-analytic with analyticity radius at least $\tau_{0}>0$, and analytic norm $K_{0} >0$. For $t \geq 0$ define the Fourier-Sobolev norm
\begin{align}
a(t) = \Vert \left( (-\Delta)^{3/4} \Theta(\cdot,s)\right)^{\hat{\ \ }}\Vert_{\ell^{1}} = \sum_{\kk \in \ZZdstar} |\kk|^{3/2} |\hat{\Th}(\kk,t)|, \label{eq:a:def}
\end{align}
and let
\begin{align*}
A(t) = \int_{0}^{t}  a(s) ds.
\end{align*}
There exists a positive  constant $C_{r}$ depending only on $r> d/2 + 5/2$, such that if at time $T$ we have
\begin{align}\label{eq:thm:BKM}
\frac{\tau_{0}}{C_{r}} > A(T) \exp\left( C_{r} K_{0} \int_{0}^{T} \exp( - A(t) )\, dt \right)
\end{align}
then there exists a real-analytic solution of the initial value problem associated to the $\AMGz$ equations on $[0,T+\delta]$, for some $\delta>0$.
\end{theorem}
The above analytic persistency criterion \eqref{eq:thm:BKM} should be viewed in the context of numerical simulations of the $\AMGz$ equations: given $K_{0}$ and $\tau_{0}$, Theorem~\ref{thm:BKM} provides a numerically easy-to-track quantity which signals the time when the simulation becomes un-reliable/under-resolved. In particular,  for most  numeric simulations the initial data is  {\em entire} (a finite trigonometric polynomial), and hence $\tau_{0}$ may be taken arbitrarily large. In this case condition \eqref{eq:thm:BKM} shows that if $A(t)$ increases at a fast rate near  $T$, the exponential factor on the right side of \eqref{eq:thm:BKM} becomes negligible compared to $A(T)$, and so the breakdown of analyticity is due to the accumulation near $T$ of the Sobolev norm.

\begin{proof}[Proof of Theorem~\ref{thm:BKM}]
With the notation of the  proof of Theorem~\label{thm:analytic}, we recall the estimate \eqref{eq:BKM:ode}, i.e.
\begin{align}
\frac{1}{2} \frac{d}{dt} \Vert \Th \Vert_{\tau}^{2} = \dot{\tau} \Vert (-\Delta)^{1/4} \Th \Vert_{L^{2}}^{2} + {\mathcal R}. \label{eq:BKM:11}
\end{align}
The proofs differ in the way we estimate nonlinear term ${\mathcal R}$. From \eqref{eq:BKM:est:1},  which is symmetric with respect to $|\jj|$ and $|\kk|$, and using the bound $e^{x} \leq 1 + x e^{x}$ which holds for all $x\geq 0$,  ${\mathcal R}$ may be bounded as
\begin{align}
{\mathcal R} &\leq C \sum_{\jj+\kk=\lb; \jj,\kk,\lb \in \ZZdstar; |\jj| \leq |\kk|} |\jj|^{3/2} |\hat{\Th}(\jj)| e^{\tau |\jj|} |\kk|^{r+1/2} |\hat{\Th}(\kk)| e^{\tau |\kk|} |\lb|^{r+1/2} |\hat{\Th}(\lb)| e^{\tau |\lb|} \notag\\
&\leq C \Vert (-\Delta)^{1/4} \Th \Vert_{\tau}^{2}  \sum_{\jj \in \ZZdstar} |\jj|^{3/2} |\hat{\Th}(\jj)| \left( 1 + \tau |\jj| e^{\tau |\jj|} \right)\notag\\
&\leq C_{r}\Vert (-\Delta)^{1/4} \Th \Vert_{\tau}^{2}  \left( \sum_{\jj \in \ZZdstar} |\jj|^{3/2} |\hat{\Th}(\jj)| + \tau \Vert \Th \Vert_{\tau} \right) \label{eq:BKM:2}
\end{align}
if $r>d/2 + 5/2$, where $C_{r} > 0$ is a large enough constant depending only on $r$. Combining \eqref{eq:BKM:11} with \eqref{eq:BKM:2}, and recalling the notation \eqref{eq:a:def}, we obtain the a priori estimate
\begin{align}
\frac{1}{2} \frac{d}{dt} \Vert \Th \Vert_{\tau}^{2} \leq \left( \dot{\tau} + C_{r} a(t) + C_{r} \tau \Vert \Th \Vert_{\tau} \right) \Vert (-\Delta)^{1/4} \Th \Vert_{L^{2}}^{2}\label{eq:BKM:3}.
\end{align}
Therefore, if we choose $\tau(t)$ decreasing fast enough so that
\begin{align}
\dot{\tau} + 3C_{r} a(t) + 2 C_{r} \tau \Vert \Th \Vert_{\tau} \leq 0, \label{eq:BKM:4}
\end{align}
then from \eqref{eq:BKM:3}, using that $\Vert \Th \Vert_{\tau} \leq \Vert (-\Delta)^{1/4} \Th \Vert_{\tau}$, we obtain
\begin{align}
\frac{d}{dt} \Vert \Th \Vert_{\tau} + C_{r} a(t) \Vert \Th \Vert_{\tau} \leq 0 \label{eq:BKM:5}
\end{align}
as long as $\tau>0$ and \eqref{eq:BKM:3} holds. Recalling that $K_{0} = \Vert \Th(\cdot,0) \Vert_{\tau_{0}}$, it follows from \eqref{eq:BKM:5} that \eqref{eq:BKM:4} holds if we let $\tau(t)$ solve the initial value problem
\begin{align}
\dot{\tau}(t) + 3C_{r} a(t) + 2 C_{r}  K_{0} \tau(t) \exp(-C_{r} A(t)) = 0\label{eq:BKM:6},
\end{align}
with initial data $\tau(0)$. This ordinary differential equation may be solved explicitly, yielding that
\begin{align}
\tau(t) &=  \tau_{0} \exp\left( - 2C_{r} K_{0} \int_{0}^{t} \exp(-C_{r}A(s)) ds \right)\notag\\
& \qquad \qquad - 3 C_{r} \int_{0}^{t} a(s) \exp\left( -2 C_{r} K_{0} \int_{s}^{t} \exp(-C_{r} A(z) dz \right) ds. \label{eq:BKM:7}
\end{align}
After a short calculation, \eqref{eq:BKM:7} shows that that as long as
\begin{align*}
\frac{\tau_{0}}{3C_{r}} > A(t) \exp\left( -2 C_{r} K_{0} \int_{0}^{t} \exp(-C_{r}A(s) ) ds \right)
\end{align*}
we have $\tau(t) >0$, and therefore the real-analytic solution may be continued a bit past $t$, thereby concluding the proof of the theorem.
\end{proof}

\begin{remark}[{\bf A proof in real variables}]\label{rem:existence:periodic}
While for the $\AMGz$ equations set on $\TT^{d}$ the Gevrey-class norms introduced in \cite{FoiasTemam} are very convenient to work with, the proofs of both Theorem~\ref{thm:analytic} and Theorem~\ref{thm:BKM} may alternatively be given in the real variables. For this purpose, one may use the real-analytic norms introduced in \cite{KukavicaTemamVicolZiane} for the study of the hydrostatic Euler equations, i.e. one may work with
\begin{align*}
\Vert \Th \Vert_{X_{\tau}} = \sum_{\alpha \geq 0} \Vert \partial^{\alpha} \Th \Vert_{L^{2}} \frac{\tau^{|\alpha|} (|\alpha|+1)^{3}}{|\alpha|!} \qquad \mbox{and} \qquad
\Vert \Th \Vert_{Y_{\tau}} = \sum_{\alpha \geq 1 } \Vert \partial^{\alpha} \Th \Vert_{L^{2}} \frac{\tau^{|\alpha|-1} (|\alpha|+1)^{3}}{(|\alpha|-1)!}
\end{align*}
for some $\tau>0$, where $\alpha \in {\mathbb N}_{0}^{3}$ is a multi-index.  Closely following the arguments in \cite{KukavicaTemamVicolZiane}, one may obtain
\begin{align}
\frac{d}{dt} \Vert \Th \Vert_{X_{\tau}} \leq \left( \dot{\tau} + C_{0} (1+ \tau^{-5/2}) \Vert \Th \Vert_{X_{\tau}} \right) \Vert \Th \Vert_{Y_{\tau}}\label{eq:analytic:ode}
\end{align}
for some constant $C_{0}>0$, as long as $\tau>0$. The estimate \eqref{eq:analytic:ode} corresponds to the bounds \eqref{eq:BKM:est:3} or \eqref{eq:BKM:3} from the periodic setting.
A suitable choice of decreasing $\tau$ then completes the proof in the real variables, as in the $\TT^{d}$ case. Although a real-variable proof is more natural, we have chosen here to present  Fourier-based proofs of all results in this section for self-consistency of the paper, and transparency of the proofs.
\end{remark}

\begin{remark}[{\bf Additional structure of the $T_{ij}$ implies local well-posedness in Sobolev spaces}]\label{rem:existence:sobolev}
If besides the fact that $\partial_{i} \partial_{j} T_{ij} f=0$ for all smooth $f$,  the matrix $\{T_{ij}\}_{i,j=1}^{d}$ {\em additionally} satisfies
\begin{align}
\langle T_{ij} f , g \rangle =  \langle  T_{ij} g , f \rangle \label{eq:T:cond:Sobolev}
\end{align}
for all smooth $f$ and $g$,  the $\AMGz$ equations are {\em locally well-posed in  Sobolev spaces}. Note that if \eqref{eq:T:cond:Sobolev} holds, the operator $\Th \mapsto \Ub$  is {\em anti-symmetric} (given by an {\em odd} Fourier symbol). In this case, the $\AMGz$ equations are locally well-posed in Sobolev spaces $H^{s}$, with $s$ large enough. The proof closely follows from the arguments  of Chae, Constantin, Cordoba, Gancedo, and Wu~\cite{CCCGW}. In particular, when estimating the $H^{s}$ norm of $\Th$ the most ``dangerous'' terms for closing the energy estimate are
\begin{align*}
{\mathcal R}_{1} = \langle \Ub \cdot \nabla (-\Delta)^{s/2}  \Th ,(-\Delta)^{s/2} \Th\rangle \qquad \mbox{and}\qquad {\mathcal R}_{2} = \langle  (-\Delta)^{s/2} \Ub \cdot \nabla \Th, (-\Delta)^{s/2} \Th\rangle.
\end{align*}
The first term ${\mathcal R}_{1}$ is identically $0$ since $\Ub$ is divergence-free. Here  one uses that $\partial_{i} \partial_{j} T_{ij} = 0$. For the second term one may write, using \eqref{eq:T:cond:Sobolev} and integration by parts,
\begin{align}
\label{eq:commutator}
{\mathcal R}_{2} &= \langle \partial_{i} T_{ij} (-\Delta)^{s/2} \Th\; \partial_{j} \Th, (-\Delta)^{s/2} \Th\rangle = -  \langle  (-\Delta)^{s/2} \Th, \partial_{i} T_{ij} ( \partial_{j} \Th\; (-\Delta)^{s/2} \Th) \rangle \notag\\
& = - {\mathcal R}_{2} + \langle (-\Delta)^{s/2} \Th, [ \partial_{j} \Th, \partial_{i} T_{ij} ] (-\Delta)^{s/2} \Th \rangle,
\end{align}
where the bracket $[\cdot,\cdot]$ denotes a commutator. The second term on the right side of \eqref{eq:commutator} is a Coifman-Meyer commutator, which is hence bounded as $\Vert \partial_{i} \partial_{j} \Th \Vert_{L^{\infty}}	 \Vert \Th \Vert_{\dot{H}^{s}}^{2}$. We omit further details and refer the reader to \cite{CCCGW}, where it was noted that the anti-symmetry of the map $\Th \mapsto \Ub$ and the divergence-free nature if $\Ub$ may be used to close Sobolev estimates, via commutators.
\end{remark}

\subsection{Hadamard ill-posedness the linearized $\MGz$ equations}
\label{sec:ill}

The classical approach to Hadamard ill-posedness for partial differential equations arising in fluid dynamics (see for instance~\cite{Gerard-VaretDormy,Gerard-VaretNguyen,Grenier,GuoNguyen,Renardy})  is to first linearize the equations about a suitable steady state $\Theta_0$, such that the linear operator $L$ obtained has eigenvalues with arbitrarily large real part, in the unstable region. In particular, this very strong instability shows  that the linear operator $L$ does not generate a semi-group:
\begin{theorem}[\bf Linear ill-posedness] \label{thm:linear:instability}
The Cauchy problem associated to the linear evolution
\begin{align}
  \partial_t \theta = L\theta \label{eq:linearequation}
\end{align} where the linear operator $L$ and the steady state $\Theta_{0}$ are given by
\begin{align}
L\theta(\xx,t) &= - M_3 \theta (\xx,t)\; \partial_3 \Theta_0(x_3) \label{eq:L:def}\\
\Theta_0(x_3) & = a \sin(m x_3) \label{eq:steadystatedef}
\end{align}for some $a>0$, and integer $m>0$, is ill-posed in the sense of Hadamard over $L^{2}$. More precisely, for any $T>0$ and any $K>0$, there exists a real-analytic initial data $\theta(0)$ such that the Cauchy problem associated to \eqref{eq:linearequation}--\eqref{eq:steadystatedef} has no solution $\theta \in L^{\infty}(0,T;L^{2})$ satisfying
\begin{align}
\sup_{t\in (0,T)} \Vert \theta(\cdot,t) \Vert_{L^{2}} \leq K \Vert \theta(0) \Vert_{Y} \label{eq:thm:boundedness}
\end{align}
where $Y$ is any Sobolev space embedded in $L^{2}$.
\end{theorem}
In order to prove Theorem~\ref{thm:linear:instability}, we need a more detailed spectral analysis of the linear operator $L$ defined in \eqref{eq:L:def}. We consider the simplest possible steady state, namely $\Ub_0 = 0$, $\Theta_0 = F(x_3)$, for some $\TT$-periodic function $F$, with $\int_{0}^{2\pi} F(x_3)\, dx_3 =0$. We denote by $\theta$ the perturbation $\Th-\Th_0$. Linearizing \eqref{eq:Ill:1} about this steady state gives
\begin{align}
  \partial_t \theta + u_3 F'(x_3)= 0, \label{eq:linear:MG}
\end{align}where $u_3 = M_3 \theta$ is given explicitly by \eqref{eq:M:3} via its Fourier coefficients
\begin{align}
  \hat{u}_3(\kk) = \frac{\mu \kky^{2}(\kkx^{2} + \kky^{2})}{4\Omega^{2} \kkz^{2} |\kk|^2 + \mu^2 \kky^{4}} \hat{\theta}(\kk) \label{eq:linear:u3}
\end{align}where we have denoted the physical parameter $\mu = \beta^2/\eta > 0$. It is then natural to consider $F$ to be a $2\pi$-periodic function consisting just of one harmonic, namely
\begin{align}
  F(x_3) = a \sin (m x_3) \label{eq:linear:f}
\end{align} for a fixed amplitude $a>0$, and an integer modulation $m>0$. Note that $F'(x_3) = m a \cos(m x_3)$.

Since the linear evolution \eqref{eq:linear:MG} does not act on the $x_1$ and $x_2$ variables, it is natural to pick initial data $\theta(\cdot,0)$ consisting of only one Fourier mode in $x_1$ and $x_2$. We hence look for  solutions to \eqref{eq:linear:MG} of the form
\begin{align}
  \label{eq:linear:ansatz}
  \theta(x,t) = e^{\sigma t} \sin(k_1 x_1) \sin(k_2 x_2) \sum_{n \geq 1} c_n \sin( n x_3)
\end{align}
for some fixed integers $k_1$ and $k_2$. In \eqref{eq:linear:ansatz}, we have chosen a sine series for convenience, since $\int_{0}^{2\pi} \theta \, dx_3= 0$. Therefore, by \eqref{eq:linear:u3} we also have
\begin{align}
  \label{eq:linear:ansatz:u}
  u_3(x,t) = e^{\sigma t} \sin(k_1 x_1) \sin(k_2 x_2) \sum_{n\geq 1}c_n \frac{ \mu \kky^{2}(\kkx^{2} + \kky^{2})}{4\Omega^{2} n^{2} (\kkx^2 + \kky^2 + n^2) + \mu^2 \kky^{4}} \sin(n x_3).
\end{align}After inserting \eqref{eq:linear:ansatz} and \eqref{eq:linear:ansatz:u} into \eqref{eq:linear:MG}, and dividing by $e^{\sigma t} \sin(k_1 x_1) \sin(k_2 x_2)$, we obtain
\begin{align}
  \sigma \sum_{n\geq 1} c_n \sin(n x_3) + a m \cos(m x_3) \sum_{n \geq 1} c_n \frac{ \mu \kky^{2}(\kkx^{2} + \kky^{2})}{4\Omega^{2} n^{2} (\kkx^2 + \kky^2 + n^2) + \mu^2 \kky^{4}}  \sin(nx_3) = 0.\label{eq:linear:recursion:1}
\end{align}To simplify the interaction between $\cos(m x_3)$ and $\sin(nx_3)$, for the value of $m$ fixed in \eqref{eq:linear:f}, we make the additional ansatz
\begin{align}
  c_n &= 0, \ \mbox{whenever}\ n\ \mbox{is not an integer multiple of}\ m, \label{eq:c:notmultiple}\\
  c_n &= c_{mp} =: \tilde{c}_p,\ \mbox{whenever}\ n=mp,\ \mbox{for some integer}\ p\geq 1. \label{eq:c:multiple}
\end{align}
Inserting \eqref{eq:c:notmultiple}--\eqref{eq:c:multiple} above into \eqref{eq:linear:recursion:1} gives
\begin{align}
  \sigma \sum_{p\geq 1} \tilde{c}_{p} \sin(mp\, x_3) + \sum_{p\geq 1}\frac{\tilde{c}_{p}}{ \alpha_{p}} \Big( \sin( m(p+1)\, x_3) + \sin ( m(p-1)\, x_3) \Big) = 0, \label{eq:eq:recursion}
\end{align}
where for ease of notation, since $k_1$ and $k_2$ are fixed throughout, we have denoted
\begin{align}
  \alpha_{p} =  \frac{2^{3}\Omega^{2} (mp)^2 ( \kkx^2 + \kky^2 + (mp)^2) + 2\mu^2 \kky^4}{a\mu m \kky^2 (\kkx^2 + \kky^2)}\label{eq:def:alpha}
\end{align}
for any $p\geq 1$, where $a,m\geq 1$ are fixed as in \eqref{eq:linear:f}. The essential feature of the $\alpha_p$ coefficients is that they grow as $p^4$ when $p\rightarrow \infty$. The above equation \eqref{eq:eq:recursion} gives the recurrence relation for the sequence $\tilde{c}_p$, in terms of the \textit{given} sequence $\alpha_p$
\begin{align}
  &\sigma \tilde{c}_p + \frac{\tilde{c}_{p+1}}{\alpha_{p+1}} + \frac{\tilde{c}_{p-1}}{\alpha_{p-1}} = 0,\ \mbox{for all}\ p\geq 2, \label{eq:c:1}\\
  & \sigma \tilde{c}_1 + \frac{\tilde{c}_2}{\alpha_2} = 0,\ \mbox{for}\ p=1.\label{eq:c:2}
\end{align}
Letting $\eta_{p} = (\tilde{c}_p \alpha_{p-1})/(\tilde{c}_{p-1} \alpha_{p})$, we obtain
\begin{align}
  &\sigma \alpha_p + \eta_{p+1} + \frac{1}{\eta_{p}} = 0,\ p\geq 2 \label{eq:eta:1}\\
  &\sigma \alpha_1 + \eta_2 = 0,\ p=1, \label{eq:eta:2}
\end{align}and therefore for any $p\geq 2$ we have
\begin{align}
  \eta_p = \frac{-1}{\sigma \alpha_p + \eta_{p+1}} = \frac{-1}{\sigma \alpha_p - \frac{1}{\sigma \alpha_{p+1} + \eta_{p+2}}} \label{eq:eta:def}
\end{align}and we obtain $\eta_{p}$ as a continued fraction. Equating with the case $p=2$ gives an equation in $\sigma$, namely
\begin{align}
  \sigma \alpha_1 = \frac{1}{\sigma \alpha_2 - \frac{1}{\sigma \alpha_3 - \frac{1}{\sigma \alpha_4 - \ldots}}}. \label{eq:therecursion}
\end{align}
The goal is to find a real, positive, solution $\sigma$ of \eqref{eq:therecursion}, since then the solution $\theta(t)$ grows exponentially in time like $\exp(\sigma t)$ (cf.~\eqref{eq:linear:ansatz}). The following lemma states the existence of such a solution.

\begin{lemma}\label{lemma:sigma}
Let $\alpha_p$ be defined for all $p\geq 1$ as in \eqref{eq:def:alpha}, where the positive integers $m,a,k_1,k_2$ are fixed, and $\mu,\Omega>0$ are fixed physical parameters. Then, there exists a real solution $\sigma_*>0$ to \eqref{eq:therecursion}, and moreover we have $\sigma_*> 1/ \sqrt{\alpha_1 \alpha_2}$.
\end{lemma}
\begin{proof}[Proof of Lemma~\ref{lemma:sigma}]
For ease of notation we define the functions $F_p$ and $G_p$ as
\begin{align}
  F_p(\sigma) &= \frac{1}{\sigma \alpha_p - \frac{1}{\sigma \alpha_{p+1} - \frac{1}{\sigma \alpha_{p+2} - \ldots}}}\label{eq:Fdef}\\
  G_p(\sigma) & = \frac{\sigma \alpha_p - \sqrt{\sigma^2 \alpha_p^2 - 4}}{2} = \frac{2}{\sigma\alpha_p + \sqrt{\sigma^2 \alpha_p^2 - 4}}\label{eq:Gdef}
\end{align}
for all $p\geq 2$, and all $\sigma > \sigma_0=2/\alpha_2$. $F_p$ is well-defined and smooth, except for a set of points on the real axis with $\sigma < \sigma_0$. For the rest of the proof of this lemma we will {\em always} assume that $\sigma > \sigma_0 \geq  2/\alpha_p$ for any $p\geq 2$. Note that $G_p$ satisfies
\begin{align*}
  G_p(\sigma) = \frac{1}{\sigma \alpha_p - G_p(\sigma)} = \frac{1}{\sigma \alpha_p - \frac{1}{\sigma \alpha_{p} - \frac{1}{\sigma \alpha_{p} - \ldots}}}
\end{align*}
for all for $p\geq 2$. Since the sequence $\alpha_p$ is strictly increasing in $p$, we have that
\begin{align}
G_2(\sigma) > G_3(\sigma) > G_4(\sigma) \ldots \geq 0. \label{eq:Gcompare}
\end{align}
Since $\lim_{p\rightarrow \infty} \alpha_p = \infty$, we have that $\lim_{p\rightarrow \infty} G_p(\sigma) = 0$ for every fixed $\sigma$. Lastly, for all $p\geq 2$ we have
\begin{align}
  \sigma \alpha_p > G_{p+1}(\sigma)\label{eq:alpha:ineq}
\end{align}
for all $\sigma > \sigma_0$. From \eqref{eq:Gdef} and \eqref{eq:Gcompare} we obtain that $\sigma \alpha_2 - G_3(\sigma) > \sigma \alpha_2 - G_2(\sigma) > 0$ and therefore
\begin{align*}
  G_2(\sigma) = \frac{1}{\sigma \alpha_2 - G_2(\sigma)} > \frac{1}{\sigma \alpha_2 - G_3(\sigma)} > 0.
\end{align*}
Similarly, to the above inequality, using \eqref{eq:Gcompare} and \eqref{eq:alpha:ineq}, we obtain
\begin{align*}
  0 < \sigma \alpha_2 - G_3(\sigma) = \sigma \alpha_2 - \frac{1}{\sigma \alpha_3 - G_3(\sigma)} < \sigma \alpha_2 - \frac{1}{\sigma \alpha_3 - G_4(\sigma)},
\end{align*}
so that
\begin{align*}
  G_2(\sigma) > \frac{1}{\sigma \alpha_2 - G_3(\sigma)} > \frac{1}{\sigma \alpha_2 - \frac{1}{\sigma \alpha_3 - G_4(\sigma)}}>0
\end{align*}
holds for all $\sigma>\sigma_0$. An inductive argument then gives that
\begin{align}
  G_2(\sigma) > \frac{1}{\sigma \alpha_2 - G_3(\sigma)} > \frac{1}{\sigma \alpha_2 - \frac{1}{\sigma \alpha_3 - G_4(\sigma)}}> \ldots > \frac{1}{\sigma \alpha_2 - \frac{1}{\sigma \alpha_3 - \ldots - \frac{1}{\sigma \alpha_p - G_{p+1}(\sigma)}}} \rightarrow F_2(\sigma)\geq 0 \label{eq:F2:upper}
\end{align}for all $p \geq 2$, and a.e. $\sigma >\sigma_0$. Repeating the above construction argument we obtain that $0 \leq F_3(\sigma) < G_3(\sigma)$, and hence by \eqref{eq:alpha:ineq} we get $0< \sigma \alpha_2 - G_3(\sigma) < \sigma \alpha_2 - F_3(\sigma) < \sigma \alpha_2 $, so that
\begin{align}
  F_2(\sigma) = \frac{1}{\sigma \alpha_2 - F_3(\sigma)} >  \frac{1}{\sigma \alpha_2}.\label{eq:F2:lower}
\end{align}Therefore, it follows that there exists a real $\sigma_* >\sigma_0 = 2/\alpha_2$ such that
\begin{align}
  \sigma_* \alpha_1 = F_2(\sigma_*),
\end{align}and moreover, from \eqref{eq:F2:upper} and \eqref{eq:F2:lower} we obtain the estimate (see Figure~\ref{fig} below)
\begin{align}
  \frac{1}{\sqrt{\alpha_1 \alpha_2}} < \sigma_* < \frac{1}{\sqrt{\alpha_1 \alpha_2 - \alpha_1 ^2}}\label{eq:thebound}
\end{align}which concludes the proof of the lemma.

\begin{figure}[!h]
  \includegraphics[width=76ex]{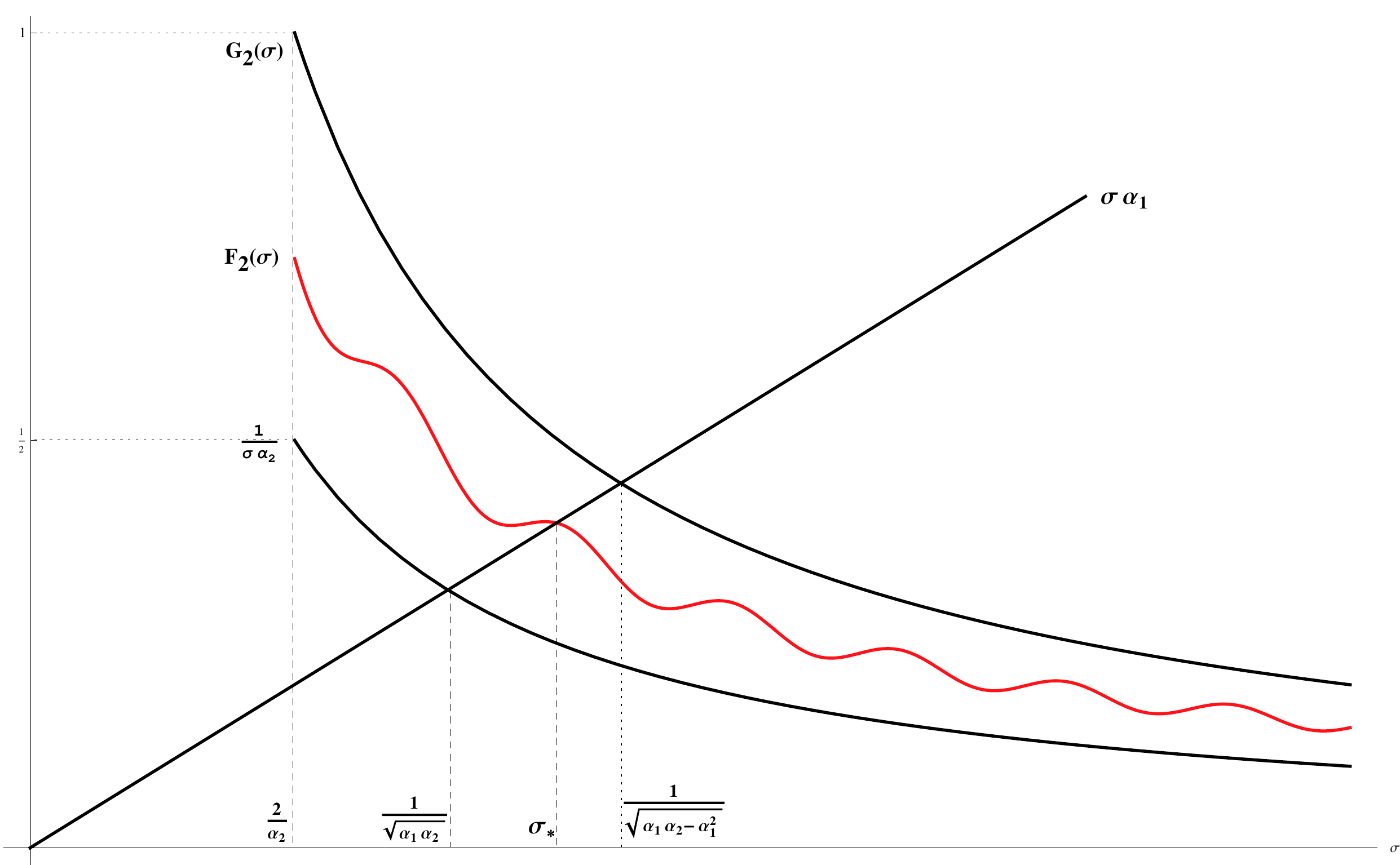}
  \caption{Diagram showing how estimates \eqref{eq:F2:upper} and \eqref{eq:F2:lower} give the bound \eqref{eq:thebound}.} \label{fig}
\end{figure}
\end{proof}

From Lemma~\ref{lemma:sigma} we therefore obtain a solution $\sigma_{\ast}$ to \eqref{eq:therecursion}. Since the coefficients $\alpha_{p}$ are given (cf.~\eqref{eq:def:alpha}), the recursion relations \eqref{eq:eta:1}--\eqref{eq:eta:2} uniquely defines the values of $\eta_{p}$. We choose
\begin{align}
\tilde{c}_{1} &= \alpha_{1},\label{eq:cp1:def}\\
\tilde{c}_{p} &= \alpha_{p}\eta_{p} \eta_{p-1} \ldots \eta_{2}, \mbox{ for all } p \geq 2. \label{eq:cp:def}
\end{align}
This sequence satisfies the recursion relations \eqref{eq:c:1}--\eqref{eq:c:2} by construction. We claim that the above defined coefficients $\tilde{c}_p$ decay very fast as $p\rightarrow \infty$. To see this, from \eqref{eq:eta:def} we observe that
\begin{align*}
 \eta_{p}=  \frac{-1}{\sigma_{\ast} \alpha_{p} + \eta_{p+1}} = \frac{-1}{\sigma_\ast \alpha_p - \frac{1}{\sigma_\ast \alpha_{p+1} + \eta_{p+2}} } = \frac{-1}{\sigma_\ast \alpha_p - \frac{1}{\sigma_\ast \alpha_{p+1} -\frac{1}{\sigma_\ast \alpha_{p+2}+ \eta_{p+3}} }} = \ldots = - F_p(\sigma_\ast).
\end{align*}
Moreover, similarly to \eqref{eq:F2:upper}--\eqref{eq:F2:lower} one may prove that $G_p(\sigma) > F_p(\sigma) > 1/(\sigma \alpha_p)$ for all $\sigma > \sigma_0$, and hence
\begin{align}
\frac{-2}{\sigma_\ast \alpha_p + \sqrt{\sigma_{\ast}^{2} \alpha_{p}^{2} - 4}}  < \eta_p < \frac{-1}{\sigma_\ast \alpha_p} \label{eq:eta:limit}.
\end{align}
Moreover, from \eqref{eq:def:alpha} we have $\alpha_{p} = {\mathcal O}(p^4)$ as $p \rightarrow \infty$, and we thus from \eqref{eq:eta:limit} above we obtain $\eta_{p} = {\mathcal O}(1/p^4)$ as $p \rightarrow \infty$.  Using \eqref{eq:eta:limit}, it follows from \eqref{eq:cp:def} that
\begin{align*}
\lim_{p\rightarrow \infty} \tilde{c}_{p} = 0
\end{align*}
and this convergence is very fast, namely of the order of  $C^{p}/((p-1)!)^{4}$ as $p\rightarrow \infty$, for some positive constant $C = C(\sigma_{\ast},\mu,\Omega, a, m,k_{1},k_{2})$. Therefore, the solution $\theta(x,t)$ of \eqref{eq:linear:f}, defined by \eqref{eq:linear:ansatz} lies in any Sobolev space, it is $C^{\infty}$ smooth, and is even real-analytic. In summary, we have proven:

\begin{lemma}\label{lemma:theta:construct}
Fix integers $a,m \geq 1$ and a Sobolev space $Y\subset L^{2}$. For any integer $j \geq m$ there exists a $C^{\infty}$ smooth initial datum $\theta^{(j)}(\cdot,0)$, with $\Vert \theta^{(j)}(\cdot,0) \Vert_{Y}=1$, and a $C^{\infty}$ smooth function $\theta^{(j)}(x,t)$ solving the initial value problem associated to the linearized $\MGz$ equations \eqref{eq:linearequation}--\eqref{eq:steadystatedef}, such that
\begin{align}
\Vert \theta^{(j)}(\cdot,t) \Vert_{L^2} \geq \exp(j\, t\, C_{a,m,\mu,\Omega})  \label{eq:lineargrowth}
\end{align}
for all $t>0$, where $C_{a,m,\mu,\Omega}>0$ is given by \eqref{eq:sigma:bound} below.
\end{lemma}
\begin{proof}[Proof of Lemma~\ref{lemma:theta:construct}]
Recalling the definition of $\alpha_p$ (cf.~\eqref{eq:def:alpha}), we have that the function $\theta(\cdot,t)$ given by \eqref{eq:linear:ansatz},  with coefficients $c_{p}$ given by \eqref{eq:cp1:def}--\eqref{eq:cp:def}, grows at the rate of $\exp(\sigma_{*}t)$, where $\sigma_*$ is such that
\begin{align*}
  \sigma_* > \frac{a \mu m \kky^2 (\kkx^2 + \kky^2)}{2^{5} \Omega^2 m^2 ( \kkx^2 + \kky^2 + 4 m^2) + 2 \mu^2 \kky^4}.
\end{align*}
Letting $(k_1,k_2) \in \ZZ^2$ be such that $k_2^2 = k_1 =j \geq m \geq 1$, we obtain that for any $a,m$, and $j$ with $j \geq m$, there exists a solution $\sigma_*$ of \eqref{eq:therecursion} bounded from below as
\begin{align}
  \sigma_* > j \frac{a\mu m}{2^{8}\Omega^{2} m^2 + 2\mu^2} =j\, C_{a,m,\mu,\Omega} \label{eq:sigma:bound}
\end{align}
concluding the proof of the lemma.
\end{proof}
In order to prove the severe type of ill-posedness stated in Theorem~\ref{thm:linear:instability}, we need to show that it is {\em not} the uniqueness of solutions to the Cauchy problem associated to \eqref{eq:linearequation} which fails. Indeed, we have the following result regarding the uniqueness of smooth solutions to the linear equation.
\begin{proposition}[{\bf Uniqueness for the linearized equations}]\label{prop:unique}
Let $\theta \in L^{\infty}(0,T;L^{2})$ be a solution of the Cauchy problem associated to \eqref{eq:linearequation}--\eqref{eq:steadystatedef}, with initial data $\theta(\cdot,0)=0$. Then for any $t\in(0,T)$ we have $\theta(\cdot,t)=0$.
\end{proposition}
\begin{proof}[Proof of Proposition~\ref{prop:unique}]
Since $\theta \in L^{\infty}_{t}L^{2}_{x}$, it can be written as the sum of its Fourier series
\begin{align*}
\theta(\xx,t) = \sum_{\kk \in \ZZ^{3}} \hat{\theta}(\kk,t) e^{i \kk \cdot \xx}.
\end{align*}
Taking the Fourier transform of \eqref{eq:linearequation}, and noting that $\partial_{3}\Th_{0}$ is a function just of $x_{3}$, we obtain (suppressing the $t$ dependence)
\begin{align*}
\partial_{t} \hat\theta(\kk) + \sum_{n\in \ZZ \setminus\{ 0,k_{3} \}} \frac{\mu \kky^{2}(\kkx^{2} + \kky^{2})}{4\Omega^{2} n^{2} (\kkx^{2} + \kky^{2} + n^{2}) + \mu^2 \kky^{4}} \hat\theta(k_{1},k_{2},n) i (k_{3}-n) \hat{\Th}_{0}(k_{3}-n) = 0
\end{align*}
where we denoted the Fourier series coefficients of $\Th_{0}$ by $\hat\Th_{0}(n)$, for all $n\in \ZZ\setminus \{0\}$. Multiplying the above equation by $\overline{\hat\theta(\kk)}$ and summing over $\kkz \in \ZZ \setminus\{0\}$, we obtain the bound
\begin{align}
&\frac 12 \frac{d}{dt} \sum_{\kkz\in \ZZ\setminus\{ 0\}} | \hat\theta(\kkx,\kky,\kkz)|^{2}\notag\\
& \qquad \leq \sum_{\kkz\in \ZZ\setminus\{ 0\}} \sum_{n\in \ZZ\setminus\{ 0,k_{3} \}} \frac{\mu \kky^{2}}{4\Omega^{2} n^{2} } |\hat\theta(k_{1},k_{2},n)| |k_{3}-n| |\hat{\Th}_{0}(k_{3}-n)| |\hat\theta(k_{1},k_{2},k_{3})|\notag\\
&\qquad \leq \frac{\mu \kky^{2}}{4 \Omega^{2}}   \Vert \partial_{3}\Th_{0} \Vert_{L^{2}}  \left(\sum_{\kkz\in \ZZ\setminus\{ 0\}} | \hat\theta(\kkx,\kky,\kkz)|^{2}\right)^{1/2} \sum_{n \in \ZZ \setminus\{0\}} |\hat\theta(\kkx,\kky,n)| \frac{1}{n^{2}}\notag\\
& \qquad \leq \frac{\mu \kky^{2}}{4 \Omega^{2}} \frac{\pi^{2}}{3 \sqrt{5}}  \Vert \partial_{3}\Th_{0} \Vert_{L^{2}} \sum_{\kkz\in \ZZ\setminus\{ 0\}} | \hat\theta(\kkx,\kky,\kkz)|^{2} ,\label{eq:uniqueness:ode1}
\end{align}
where in the second inequality we reversed the order of summation and used the Cauchy-Schwartz inequality in the $\kkz$ summation, and in the third inequality we used that $\sum_{n\geq 1 } n^{-4} = \pi^{4}/90$. Note that in \eqref{eq:uniqueness:ode1} we do not sum in $\kky$, nor in $\kkx$. Using Gr\"onwall's inequality, and the initial condition $\hat\theta(0,\kk)=0$ for all $\kk \in \ZZ^{3}$, we obtain from \eqref{eq:uniqueness:ode1}
\begin{align}
\sum_{\kkz\in \ZZ\setminus\{ 0\}} | \hat\theta(t,\kkx,\kky,\kkz)|^{2} =0 \label{eq:uniqueness:proof}
\end{align}for any {\em fixed} $\kkx,\kky \in \ZZ$, and all $t \in (0,T)$. Since $\kkx$ and $\kky$ are arbitrary, it follows from \eqref{eq:uniqueness:proof} that $\hat\theta(t,\kk) = 0$ for all $\kk \in \ZZ^{3}$, with $\kkz \neq 0$, and all $t \in (0,T)$, concluding the proof of the proposition.
\end{proof}

We now have all the ingredients needed to prove Theorem~\ref{thm:linear:instability}. The guiding principle behind this result is that the existence of eigenfunctions for $L$, with arbitrarily large eigenvalues in the unstable region, prevents the existence of a continuous solution map, and in particular the problem is hence \textit{ill-posed} in the sense of Hadamard.

\begin{proof}[Proof of Theorem~\ref{thm:linear:instability}]
Let $T>0$ and $K>0$ be arbitrary, and let $\theta(\cdot,0)\in Y \subset L^{2}$ . Assume by contradiction that there exists a solution $\theta\in L^{\infty}(0,T;L^{2})$ of the Cauchy problem associated to \eqref{eq:linearequation}--\eqref{eq:steadystatedef} such that \eqref{eq:thm:boundedness} holds, i.e.
\begin{align}
\sup_{t \in [0,T]} \Vert \theta(\cdot,t) \Vert_{L^{2}} \leq K \Vert \theta(0) \Vert_{Y}.\label{eq:ill:proof:bound}
\end{align} Note that this solution also satisfies $\partial_{t} \theta \in L^{\infty}(0,T;H^{-1})$ so that $\theta$ is weakly continuous with values in $L^{2}$. On the other hand, for any $j \geq m$, from Lemma~\ref{lemma:theta:construct} we obtain the existence of a smooth solution $\theta^{(j)}$ of \eqref{eq:linearequation}, with $\Vert \theta^{(j)}(\cdot,0)\Vert_{Y}=1$,  which by \eqref{eq:lineargrowth} satisfies
\begin{align}
\Vert \theta^{(j)}(\cdot,t) \Vert_{L^{2}} \geq \exp(j\, T\, C_{a,m,\mu,\Omega}/2),\label{eq:ill:proof:lowerbound}
\end{align}
for all $t\in(T/2,T)$. Moreover, this solution is unique in $L^{\infty}_{t}L^{2}_{x}$ by Proposition~\ref{prop:unique}. Letting $j$ be a sufficiently large integer, for instance choosing
\begin{align*}
j = \max\left\{ \frac{2 \log(1+ K)}{T C_{a,m,\mu,\Omega}},m\right\},
\end{align*}
we obtain from \eqref{eq:ill:proof:lowerbound} a contradiction with \eqref{eq:ill:proof:bound}, thereby proving the theorem.
\end{proof}

\subsection{Hadamard ill-posedness in Sobolev spaces for the nonlinear $\MGz$ equation}
We first make precise the definition of Lipschitz well-posedness for the full {\em nonlinear} $\MGz$ equation (see~\cite[Definition 1.1]{GuoNguyen}).

\begin{definition}[\bf Lipschitz local well-posedness] \label{def:well}
Let $Y \subset X \subset W^{1,4}$ be Banach spaces. The Cauchy problem for the $\MGz$ equation
\begin{align}
&\partial_{t} \Th + \Ub \cdot \nabla \Th = 0 \label{eq:nonlinear:ill:1}\\
&\nabla \cdot \Ub =0, \  U_{j} =  M_{j} \Th \label{eq:nonlinear:ill:2}
\end{align}is locally Lipschitz $(X,Y)$ well-posed, if there exist continuous functions $T,K: [0,\infty)^{2} \rightarrow (0,\infty)$,  the time of existence and the Lipschitz constant, so that for every pair of initial data $\Th^{(1)}(\cdot,0), \Th^{(2)}(\cdot,0) \in Y$ there exist unique solutions $\Th^{(1)}, \Th^{(2)} \in L^{\infty}(0,T;X)$ of the initial value problem associated to the $\MGz$ equation \eqref{eq:nonlinear:ill:1}--\eqref{eq:nonlinear:ill:2}, that additionally satisfy
\begin{align}
\Vert \Th^{(1)}(\cdot,t) -\Th^{(2)}(\cdot,t) \Vert_{X} \leq K \Vert \Th^{(1)}(\cdot,0) - \Th^{(2)}(\cdot,0) \Vert_{Y}\label{eq:def:well}
\end{align}
for every $t\in [0,T]$, where $T = T(\Vert \Th^{(1)}(\cdot,0) \Vert_{Y}, \Vert \Th^{(2)}(\cdot,0)\Vert_{Y})$ and $K  = K (\Vert \Th^{(1)}(\cdot,0)\Vert_{Y},\Vert \Th^{(2)}(\cdot,0)\Vert_{Y})$.
\end{definition}

If $\Th^{(2)}(\cdot,t)=0$ and $X=Y$, the above definition recovers the usual definition of local well-posedness with a continuous solution map. However, Defintion~\ref{def:well} allows the solution map to lose regularity. In~\cite{GuoNguyen} the above defined well-posedness is referred to as {\em weak well-posedness}.  It has been pointed out to us by Benjamin Texier that a loss in regularity is usually needed in order to obtain Lipschitz continuity of the solution map for quasi-linear equations, and hence the typical spaces $(X,Y)$ we have in mind are $X = H^{s}$, and  $Y=H^{s+1}$, with $s>1 + d/4$.  For the purpose of our ill-posedness result, we shall let $\Th^{(2)}(\xx,t)$ be the steady state $\Th(x_{3})$ introduced earlier in \eqref{eq:steadystatedef}. We consider $X$ to be a Sobolev space with high enough regularity so that $\partial_{t} \Th \in L^{\infty}(0,T;L^{2})$, which implies that $\Th$ is weakly continuous on $[0,T]$ with values in $X$, making sense of the initial value problem associated to \eqref{eq:nonlinear:ill:1}--\eqref{eq:nonlinear:ill:2}.

The main result of this subsection is the following theorem.

\begin{theorem}[\bf Nonlinear ill-posedness in Sobolev spaces] \label{thm:nonlinear:illposedness}
The $\MGz$ equations are locally Lipschitz $(X,Y)$ ill-posed in Sobolev spaces $Y\subset X$ embedded in $W^{1,4}$, in the sense of Definition~\ref{def:well} above.
\end{theorem}

The proof of Theorem~\ref{thm:nonlinear:illposedness} follows from the strong {\em linear} ill-posedness obtained in Theorem~\ref{thm:linear:instability} combined with the uniqueness of solutions to the linearized equations proven in Proposition~\ref{prop:unique}, and a fairly generic perturbative argument (cf.~\cite[pp.~183]{Tao}). This program has  been successfully used in the context of the hydrostatic Euler and Navier-Stokes equations~\cite{Renardy}, or the Prandtl equations~\cite{Gerard-VaretNguyen,Grenier,GuoNguyen}, and others.

\begin{proof}[Proof of Theorem~\ref{thm:nonlinear:illposedness}]
Let $Y \subset X \subset W^{1,4}$ be as in the statement of the theorem. Since $X$ embeds in $H^{1}$, the linearized operator $L \Th = - M_{3}\Th\, \partial_{3} \Th_{0}$ maps $X$ continuously into $L^{2}$, and since $X \subset W^{1,4}$, the nonlinearity $N\Th = - M_{j}\Th\, \partial_{j} \Th$ be bounded as  $\Vert N \Th \Vert_{L^{2}} \leq \Vert \nabla \Th \Vert_{L^{4}}^{2} \leq C \Vert \Th \Vert_{X}^{2}$, for some  constant $C>0$. For instance, one may consider $X=H^{s}$ and $Y=H^{s+1}$, where $s > 7/4$ in three dimensions.

Fix the steady state $\Th_{0}(x_{3}) \in Y$, as given by \eqref{eq:steadystatedef}. Also, fix a smooth function $\psi_{0} \in Y$, normalized to have $\Vert \psi_{0} \Vert_{Y} = 1$, to be chosen precisely later. The proof is by contradiction. Assume that the Cauchy problem for the $\MGz$ equation \eqref{eq:nonlinear:ill:1}--\eqref{eq:nonlinear:ill:2} were Lipschitz locally well-posed in $(X,Y)$. Consider $\Th^{(2)}(\xx,0) = \Th_{0}(x_{3})$, so that $\Th^{(2)}(\xx,t) = \Th_{0}(x_{3})$ for any $t>0$. Also let 
\begin{align*}
\Th^{\epsilon}(\xx,0) = \Th_{0}(x_{3})  + \epsilon \psi_{0}(\xx),
\end{align*} for every $0 < \epsilon < \Vert \Th_{0} \Vert_{Y}$. To simplify notation we write $\Th^{\epsilon}$ instead of $\Th^{(1,\epsilon)}$. By definition~\ref{def:well}, for every $\epsilon$ as before there exists a positive time $T = T (\Vert \Th_{0}\Vert_{Y},\Vert \Th^{\epsilon}\Vert_{Y})$ and a positive Lipschitz constant $ K = (\Vert \Th_{0}\Vert_{Y},\Vert \Th^{\epsilon}\Vert_{Y})$ such that by \eqref{eq:def:well} and the choice of $\psi_{0}$ we have
\begin{align}
\Vert \Th^{\epsilon}(\cdot,t) - \Th_{0}(\cdot) \Vert_{X} \leq K \epsilon \label{eq:Lip:bound}
\end{align}for all $t\in [0,T]$. We note that since $\Vert \Th^{\epsilon}(\cdot,0) \Vert_{Y} \leq \Vert \Th_{0} \Vert_{Y} + \epsilon \leq 2 \Vert \Th_{0} \Vert_{Y}$, due to the continuity of $T$ and $K$ with respect to the second coordinate, we may choose $K = K(\Vert \Th_{0} \Vert_{Y}) >0$ and $T =T (\Vert \Th_{0} \Vert_{Y})>0$ independent of $\epsilon \in (0,\Vert \Th_{0} \Vert_{Y})$, such that  \eqref{eq:Lip:bound} holds on $[0,T]$.

The main idea is to write the solution $\Th^{\epsilon}$ as a perturbation of $\Theta_0$, i.e. define
\begin{align*}
\theta^{\epsilon}(\xx,t) = \Th^{\epsilon}(\xx,t) - \Theta_0(x_3)
\end{align*}
for all $t \in [0,T]$ and all $\epsilon$ as before. The initial value problem solved by $\theta^{\epsilon}$ is therefore
  \begin{align*}
    &\partial_t \theta^\epsilon = L \theta^\epsilon + N \theta^\epsilon\\
    &\theta^{\epsilon}(\cdot,0) = \epsilon \psi_{0},
  \end{align*}
where we recall that the linear and nonlinear operators are given by
  \begin{align*}
    L \theta^\epsilon &= - M_3 \theta^\epsilon\, \partial_3 \Theta_0\\
    N \theta^\epsilon &= - M_{j} \theta^\epsilon \, \partial_{j} \theta^\epsilon
  \end{align*}
and $M_{j}$ is defined by its Fourier symbol cf.~\eqref{eq:M:1}--\eqref{eq:M:3} for $j\in \{1,2,3\}$. It follows from \eqref{eq:Lip:bound} that the function
  \begin{align*}
  \psi^\epsilon(\xx,t)= \theta^\epsilon(\xx,t) / \epsilon
  \end{align*}
  is uniformly bounded with respect to $\epsilon$ in $L^{\infty}(0,T;X)$. Therefore, there exists a function $\psi$, the weak-$*$ limit of $\psi^{\epsilon}$ in $L^\infty(0,T;X)$. Note that $\psi^{\epsilon}$ solves the Cauchy problem
  \begin{align}
    \partial_t \psi^\epsilon &= L \psi^\epsilon + {\epsilon} N(\psi^\epsilon) \label{eq:psieps}\\
    &\psi^{\epsilon}(\cdot,0) = \psi_{0}.
  \end{align}
  Due to the choice of $X$, we have the bound
  \begin{align}
  \Vert N \psi^{\epsilon} \Vert_{L^{2}} \leq C \Vert \psi^{\epsilon} \Vert_{X}^{2} \leq C K^{2} \label{eq:Nbound},
  \end{align}and from \eqref{eq:psieps} we obtain that $\partial_{t} \psi^{\epsilon}$ is uniformly bounded with respect to $\epsilon$ in $L^{\infty}(0,T;L^{2})$. Therefore the convergence $\psi^{\epsilon} \rightarrow \psi$ is strong when measured in $L^{2}$. Sending $\epsilon$ to $0$ in \eqref{eq:psieps}, and using \eqref{eq:Nbound}, it follows that
  \begin{align*}
    \partial_t \psi &= L \psi\\
    \psi(\xx,0)&= \psi_{0}
  \end{align*}
  holds in $L^\infty(0,T;L^2)$, and this solution is unique due to Proposition~\ref{prop:unique}. In addition, the solution $\psi$ inherits from \eqref{eq:Lip:bound} the upper bound
  \begin{align}
  \Vert \psi(\cdot,t) \Vert_{L^{2}} \leq K \label{eq:Lip:2}
  \end{align} for all $t\in[0,T]$. But this is a contradiction with Theorem~\ref{thm:linear:instability}. Indeed, due to the existence of eigenfunctions for the linearized operator with arbitrarily large eigenvalues, one may choose $\psi_{0}$ (as in Lemma~\ref{lemma:theta:construct}) to yield a large enough eigenvalue so that in time $T/2$ the solution grows to have $L^{2}$ norm larger than $2K$, therefore contradicting \eqref{eq:Lip:2}.
\end{proof}

\section{The critically diffusive equations} \label{sec:nonlinearinstability}
In this section we consider the dissipative equations, where $\kappa >0$ is fixed. First we prove for the $\AMGk$ equations that linear implies nonlinear instability, via a bootstrap argument that is a combination of arguments given in \cite{FriedlanderPavlovicShvydkoy,FriedlanderPavlovicVicol}, and a new a priori bound obtained in the appendix. We then turn to the specific $\MGk$ equations, for which we explicitly construct an unstable eigenvalue of the linear operator obtained when linearizing about a specific steady state.

\subsection{Linear instability implies nonlinear instability}\label{sec:linearimpliesnonlinear}
For a smooth source term $S \in C^{\infty}$ we consider an arbitrary smooth steady state $\Th_{0} \in C^{\infty}$ of \eqref{eq:1}--\eqref{eq:2}, i.e. a solution of
 \begin{align}
&\Ub_{0} \cdot \nabla \Th_{0} = S\label{eq:steady:1}+ \kappa \Delta \Th_{0} ,
\end{align}where the divergence-free steady state velocity is given by
\begin{align}
\Ub_{0j} = \partial_{i} T_{{ij}} \Th_{0}. \label{eq:steady:2}
\end{align}
Writing the $\AMGk$ equations \eqref{eq:1}--\eqref{eq:2} in perturbation form, we obtain that the evolution of the perturbation buoyancy $\theta = \Th-\Th_{0}$ is governed by the initial value problem associated to
\begin{align}
&\partial_{t} \theta = L\theta + N\theta \label{eq:linearized:1},
\end{align}where the {\em dissipative} linear operator $L$ is defined as
\begin{align}
L \theta = - \Ub_{0} \cdot \nabla \theta - \uu \cdot \nabla \Th_{0} + \kappa \Delta \theta \label{eq:L}
\end{align}with $u_{j} = \partial_{i} T_{ij}\theta =  U_{j} -U_{0j}$, and the nonlinear operator $N$ is given by
\begin{align}
N \theta  = - \uu \cdot \nabla \theta = - \nabla \cdot (\uu \theta).\label{eq:N}
\end{align}
In Section~\ref{sec:diffusiveunstable} below we prove that the linear operator $L$ has eigenvalues in the unstable region in the specific case that $T_{ij} =  -\partial_{i} (-\Delta)^{-1} M_{j}$. Here we recall a suitable version of Lyapunov stability (cf.~\cite{FriedlanderPavlovicShvydkoy,FriedlanderPavlovicVicol,FriedlanderStraussVishik}).

\begin{definition}[\bf Nonlinear instability]\label{def:nonlinear:instability}
Let $(X,Z)$ be a pair of Banach spaces. A steady state $\Th_0$ is called $(X,Z)$ {nonlinearly stable} if for any $\rho>0$, there exists $\tilde{\rho}>0$ so that if
$\theta(\cdot,0) \in X$ and $\Vert{\theta(\cdot,0)}\Vert_{Z} < \tilde{\rho}$, then we
have
\begin{enumerate}
  \renewcommand{\labelenumii}{\roman{enumii}}
  \item there exists a global in time solution $\theta$ to the initial value problem \eqref{eq:linearized:1}--\eqref{eq:N} with
  $\theta \in C([0,\infty);X)$;
  \item and we have the bound $\Vert{\theta(\cdot,t)}\Vert_{Z} < \rho$ for
  a.e.~$t\in [0,\infty)$.
\end{enumerate}
An equilibrium $\Th_{0}$ that is not stable (in the above sense) is
called {Lyapunov unstable}.
\end{definition}

The Banach space $X$ is the space where a local existence theorem for the nonlinear equations is available, while $Z$ is the space where the spectrum of the linear operator is analyzed, and where the instability is measured. For the $\AMGk$ equations we measure the instability in the natural space $Z = L^{2}(\TT^{d})$, while for the space of local existence we may for instance consider $X = H^{s}(\TT^{d})$, for any $s>d/2$. Note that $H^{s}$ with $s>d/2$ is subcritical for the natural scaling of the equations, and hence obtaining the local existence of a unique $H^{s}$ solution is not difficult.  Now we are ready to formulate the main result of this section.

\begin{theorem}[\bf Linear implies nonlinear instability] \label{thm:linear:nonlinear}
 Let $\Th_0$ be a smooth, mean zero steady state solution of
 the critically dissipative $\AMGk$ equations, i.e., it solves
 \eqref{eq:steady:1}-\eqref{eq:steady:2}. If the associated linear operator
 $L$, as defined in \eqref{eq:L}, has spectrum in the unstable
 region, then the steady state is $(H^s,L^2)$ Lyapunov nonlinearly
 unstable, for arbitrary $s>d/2$.
\end{theorem}

In order to prove Theorem~\ref{thm:linear:nonlinear}, fix a linearly unstable smooth eigenfunction $\phi$, with eigenvalue of \emph{maximal} real part $\lambda$. Fix a parameter $\delta \in (0,C_{\lambda})$, where $C_{\lambda}>0$ is to be determined. It is convenient to consider the operator
\begin{align}
L_{\delta} = L - (\lambda + \delta) I
\end{align}which is a shift of $L$ that moves the spectrum to the left of the line ${\rm Im} z = - \delta$ in the complex plane, so that the resolvent of $L_{\delta}$ contains the full right-half of the complex plane. Similarly to the Navier-Stokes equations \cite{FriedlanderPavlovicShvydkoy}, the operator $L_{\delta}$ is a bounded lower order perturbation of the Laplacian (since both $\uu$ and $\nabla \theta$ only lose one derivative with respect to $\theta$). Therefore, $L_{\delta}$ generates a bounded (due to the shift) analytic semigroup $e^{L_{\delta} t}$ over any Sobolev space, and in particular over $L^{2}$. In consequence, we have
\begin{align}
\Vert L_{\delta}^{\alpha} e^{L_{\delta} t }\Vert_{{\mathcal L}(L^{2})} \leq M t^{-\alpha} \label{eq:L:semigroup}
\end{align}for any $\alpha >0 $, and for any $t>0$, where $M$ is a positive constant. Also, since $L_{\delta}$ is a bounded perturbation of the Laplacian, similarly to \cite{FriedlanderPavlovicShvydkoy} for a bounded perturbation of the Stokes operator,  we have
\begin{align}
\Vert L_{\delta}^{-\alpha} v \Vert_{L^{2}} \leq C \Vert (-\Delta)^{-\alpha} f \Vert_{L^{2}} \leq C \Vert f \Vert_{L^{2d/(d+ 4 \alpha)}} \label{eq:L:smoothing}
\end{align}
for some positive constant $C$, $\alpha \in [0,d/4)$, and all smooth, mean-free functions $f$. Note that in the last inequality of the above estimate we have also used the classical Riesz potential estimate.

We now turn to proving Theorem~\ref{thm:linear:nonlinear}. In order to achieve this, it is sufficient to prove that the trivial solution $\theta= 0$ of \eqref{eq:linearized:1} is ($H^{s},L^{2}$) nonlinearly unstable. This is achieved by considering a family of solutions
$\theta^\epsilon$ to \eqref{eq:linearized:1}, with suitable size-$\epsilon$ initial data, i.e.
\begin{align}
&\partial_t \theta^\epsilon = L \theta^\epsilon + N(\theta^\epsilon),\label{eq1}\\
  &\theta^\epsilon|_{t=0} = \epsilon \phi,\label{eq2}
\end{align}
where $\phi$ is an eigenfunction of $L$ associated
with the eigenvalue of maximal positive real part $\lambda$. We will prove the following proposition which clearly implies Theorem~\ref{thm:linear:nonlinear}.

\begin{proposition}\label{prop:critical}
There exist positive constants $C_{\ast}$ and $\epsilon_{\ast}\leq 1$, such that for every $\epsilon \in (0,\epsilon_{\ast})$, there exists $T_\epsilon > 0$ with
$\Vert {\theta^\epsilon(T_\epsilon)}\Vert_{L^2} \geq C_{\ast}$.
\end{proposition}
In summary, the above proposition states that an ${\mathcal O}(\epsilon)$ perturbation of the trivial solution leads in finite time to an ${\mathcal O}(1)$ perturbation, where $\epsilon$ is arbitrarily small.
We remark that if $\theta^{\epsilon}$ solves \eqref{eq1}--\eqref{eq2}, then $\Th^{\epsilon} = \Th_{0} + \theta^{\epsilon}$ solves \eqref{eq:1}--\eqref{eq:2}, with initial data $\Th_{0} + \epsilon \phi \in C^{\infty}$. The main difficulty is to control the nonlinearity, so that the exponential growth arising from the linear part of the equations is not canceled by nonlinear effects. As in \cite{FriedlanderPavlovicVicol} for the critical SQG equations, this is achieved by making use of an a priori bound on a sub-critical norm of $\Theta$. In the case of the $\AMGk$ equations the following estimate turns out to be sufficient for proving Proposition~\ref{prop:critical}.
\begin{lemma}\label{lemma:higher:regularity}
There exists a positive constant $C^\ast = C^\ast(\Theta_0,\phi,S,\epsilon_\ast) >0$ such that
\begin{align}
 \Vert \Delta \theta^{\epsilon}(\cdot,t) \Vert_{L^{2}} \leq C^{\ast},\ \mbox{ for all } \epsilon \in (0,\epsilon_{\ast}) \label{eq:the:apriori}
\end{align}and arbitrary $t\geq 0$.
\end{lemma}
The proof of Lemma~\ref{lemma:higher:regularity} does not follow directly from our earlier works~\cite{FriedlanderVicol,FriedlanderVicol2} on the global well-posedness and higher regularity of the $\AMGk$ system, and we give the proof in Appendix~\ref{appendix} below.

\begin{proof}[Proof of Proposition~\ref{prop:critical}]
The proof is a hybrid of the proofs for the corresponding results for the Navier-Stokes~\cite{FriedlanderPavlovicShvydkoy} and the critically dissipative SQG equation~\cite{FriedlanderPavlovicVicol}. It follows the bootstrap argument considered previously for the equations of fluid mechanics, see for example~\cite{BardosGuoStrauss,FriedlanderPavlovicShvydkoy,FriedlanderPavlovicVicol}. The argument consists of writing the mild formulation of \eqref{eq1}--\eqref{eq2}, that is
\begin{align}
\theta^{\epsilon}(t) &= \epsilon e^{Lt} \phi + \int_{0}^{t}e^{L(t-s)} N\theta^{\epsilon}(s)\, ds\notag\\
&= \epsilon e^{Lt} \phi + \int_{0}^{t} e^{(\lambda+\delta)(t-s)} e^{L_{\delta}(t-s)} N\theta^{\epsilon}(s)\, ds \notag\\
&= \epsilon e^{Lt} \phi + \int_{0}^{t} e^{(\lambda+\delta)(t-s)} L_{\delta}^{\alpha}e^{L_{\delta}(t-s)} L_{\delta}^{-\alpha}N\theta^{\epsilon}(s)\, ds \label{eq:duhamel:1}
\end{align} where $\alpha \in (1/2, 1)$ is arbitrary, to be chosen later. Applying the $L^{2}$ norm to \eqref{eq:duhamel:1} we obtain from \eqref{eq:L:semigroup} and \eqref{eq:L:smoothing} (with $d=3$), that
\begin{align}
\Vert \theta^{\epsilon} (t) \Vert_{L^{2}} &\leq C_{\phi} \epsilon e^{\lambda t} + \int_{0}^{t} e^{(\lambda+\delta)(t-s)} \Vert L_{\delta}^{\alpha}e^{L_{\delta}(t-s)} \Vert_{{\mathcal L}(L^{2})} \Vert L_{\delta}^{-\alpha}N\theta^{\epsilon}(s) \Vert_{L^{2}}\, ds\notag\\
&\leq C_{\phi} \epsilon e^{\lambda t} + C M \int_{0}^{t} e^{(\lambda+\delta)(t-s)} \frac{1}{(t-s)^{\alpha}} \Vert L_{\delta}^{-\alpha} \nabla \cdot( \uu^{\epsilon}(s) \theta^{\epsilon}(s) )\Vert_{L^{2}}\, ds\notag\\
& \leq C_{\phi} \epsilon e^{\lambda t} + C M \int_{0}^{t} e^{(\lambda+\delta)(t-s)} \frac{1}{(t-s)^{\alpha}} \Vert L_{\delta}^{1/2-\alpha} ( \uu^{\epsilon}(s) \theta^{\epsilon}(s) )\Vert_{L^{2}}\, ds\notag\\
& \leq C_{\phi} \epsilon e^{\lambda t} + C M \int_{0}^{t} e^{(\lambda+\delta)(t-s)} \frac{1}{(t-s)^{\alpha}} \Vert  \uu^{\epsilon}(s) \theta^{\epsilon}(s) \Vert_{L^{6/(4\alpha+1)}}\, ds \label{eq:duhamel:2}.
\end{align}
By the H\"older inequality and interpolation of Lebesgue spaces, we have
\begin{align}
\Vert  \uu^{\epsilon}(s) \theta^{\epsilon}(s) \Vert_{L^{6/(4\alpha+1)}} &\leq C \Vert \theta^{\epsilon} \Vert_{L^{2}}  \Vert \uu^{\epsilon} \Vert_{L^{6/(4\alpha-2)	}}\notag\\
&\leq C \Vert \theta^{\epsilon} \Vert_{L^{2}}  \Vert \nabla \theta^{\epsilon} \Vert_{L^{6/(4\alpha-2)}}\notag\\
&\leq C \Vert \theta^{\epsilon} \Vert_{L^{2}}^{\alpha + 1/4} \Vert \Delta \theta^{\epsilon} \Vert_{L^{2}}^{7/4-\alpha} \label{eq:interpolate}
\end{align}where the last inequality holds  for any $\alpha\in(3/4,1)$, since in three dimensions $L^6$ is the largest Lebesgue space that embeds in $H^1$. From here on we may fix a value of $\alpha \in(3/4,1)$ and for ease of notation we let henceforth $\alpha = 7/8$. Therefore, inserting estimate \eqref{eq:interpolate} into the bound \eqref{eq:duhamel:2}, and recalling \eqref{eq:the:apriori}, we obtain
\begin{align}
\Vert \theta^{\epsilon} (t) \Vert_{L^{2}} &\leq C_{\phi} \epsilon e^{\lambda t} + 	C (C^\ast)^{7/8} \int_0^t e^{(\lambda+\delta)(t-s)} \frac{1}{(t-s)^{7/8} } \Vert \theta^\epsilon(s) \Vert_{L^2}^{9/8} \, ds.	\label{eq:Duhamel}
\end{align}Since the exponent of the $L^2$ norm of $\theta^\epsilon$ in the above integral is larger than $1$, one may conclude the proof of the lemma, by following \emph{mutatis mutandis} the proof we gave for the critical SQG equation in \cite[Proposition 4.1]{FriedlanderPavlovicVicol}. For the sake of completeness we give here a few details.  For $R >  C_\phi := \Vert{\phi}\Vert_{L^2}$ to be chosen later, let
$T=T(R,\varepsilon)$ be the maximal time such that
\begin{align}\label{eq:Tdef}
\Vert \theta^{\epsilon}(t)\Vert_{L^2} \leq  \epsilon R e^{\lambda t},\quad\
\mbox{for}\ t\in[0,T].
\end{align}
Clearly $T\in (0,\infty]$ due to the strong continuity in $L^2$ of $t\mapsto \theta(t)$ and the chosen initial condition. Therefore, letting $\delta \in (0,\lambda/8)$, we obtain from \eqref{eq:Duhamel} and \eqref{eq:Tdef} that on  $[0,T]$ we have
\begin{align}
\Vert \theta^{\epsilon} (t) \Vert_{L^{2}} &\leq C_{\phi} \epsilon e^{\lambda t} + C_1 \left(\epsilon R e^{\lambda t}\right)^{9/8}\label{eq:Duhamel:end}
\end{align}where $C_1 = C_1(C^\ast,\lambda,\delta)>0$ is a sufficiently large constant. Using the definition \eqref{eq:Tdef} and the above estimate, we obtain the following bound on $T$:
\begin{align}\label{eq:Test}
\epsilon e^{\lambda T} \geq C_{2}:= \left( \frac{R-C_{\phi}}{C_{1} R^{9/8}}\right)^{8},
\end{align}since we have chosen $R > C_{\phi}$. Therefore we obtain that
\begin{align}
T \geq T_{\epsilon} = \frac{\log (C_{2}/\epsilon)}{\lambda}.\label{eq:T:ESTIMATE}
\end{align}To find a lower bound on $\Vert \theta^{\epsilon}(T_{\epsilon}) \Vert_{L^{2}}$ one again uses the Duhamel formula bound \eqref{eq:Duhamel:end} and the triangle inequality
\begin{align*}
\Vert \theta^{\epsilon}(T_{\epsilon}) \Vert_{L^{2}} \geq  C_{\phi} \epsilon e^{\lambda T_{\epsilon}} - C_1 \left(\epsilon R e^{\lambda T_{\epsilon}}\right)^{9/8} = C_{\phi} C_{2} - C_{1}R^{9/8} C_{2}^{9/8} = C_{2}\left( 2C_{\phi} -R \right)=:C_{\ast}
\end{align*}
thereby concluding the proof of Proposition~\ref{prop:critical}, upon letting $C_{\phi} < R < 2 C_{\phi}$.
\end{proof}

\subsection{Example of an unstable eigenvalue for the $\MGk$ equations} \label{sec:diffusiveunstable}
In this section we construct a linearly unstable eigenvalue for the critically dissipative $\MGk$ system. This shows that the assumptions of Theorem~\ref{thm:linear:nonlinear} are satisfied for instance by the specific steady state constructed here.

As in the non-diffusive case, we consider the steady state $\Theta_0 = F(x_3)$ and $\Ub_0 = 0$, where the source term is $S = - \kappa \Delta \Th_0$. The example we work with is similar to Section~\ref{sec:ill-posedness}, namely we let $F(x_3) = a \sin (m x_3)$, with $S = \kappa a m^{2} \sin (m x_3)$, for some positive integers $a,m \geq 1$ to be determined below, in terms of $\kappa$. We consider the linear evolution of the perturbation temperature  $\theta =\Th- \Theta_0$, and as in \eqref{eq:linear:ansatz} we make the assumption that $\theta$ is of the form
\begin{align}
\theta(x,t) = e^{\sigma t} \sin(k_{1} x_{1}) \sin (k_{2}x_{2}) \sum_{n \geq 1} c_{n} \sin(n x_{3})\label{eq:diff:theta}
\end{align}which by \eqref{eq:intro:MG:2} implies that the perturbation velocity is
\begin{align}
u_{3}(x,t)  = e^{\sigma t} \sin(k_{1}x_{1}) \sin(k_{2}x_{2}) \sum_{n\geq 1} c_{n} \frac{\mu k_{2}^{2} (k_{1}^{2} + k_{2}^{2}) }{4 \Omega^{2} n^{2} (k_{1}^{2} + k_{2}^{2} + n^{2}) + \mu^{2} k_{2}^{4}} \sin(n x_{3}).\label{eq:diff:u}
\end{align}
Inserting \eqref{eq:diff:theta}--\eqref{eq:diff:u} into the linearized $\MGk$ equation, we obtain
\begin{align*}
&\sum_{n\geq 1} (\sigma + \kappa (k_{1}^{2} + k_{2}^{2} + n^{2}) )c_{n} \sin(n x_{3})\\
& \qquad  \qquad + a m \cos(m x_{3}) \sum_{n \geq 1}c_{n} \frac{\mu k_{2}^{2} (k_{1}^{2} + k_{2}^{2}) }{4 \Omega^{2} n^{2} (k_{1}^{2} + k_{2}^{2} + n^{2}) + \mu^{2} k_{2}^{4}} \sin(n x_{3})=0.
\end{align*}
We now follow the construction of an unstable eigenvalue given in Section~\ref{sec:ill} for the case where $\kappa = 0$. The difference in the analysis is that $\sigma$ is
replaced by $\sigma_p$, where
\begin{align}
\sigma_{p} = \sigma + \kappa (k_{1}^{2} + k_{2}^{2} + m^{2} p^{2}) \label{eq:sigma:p}
\end{align}
for all $p \geq 1$. A suitable modification of the method of continued fractions (see also \cite{MeshalkinSinai}) again produces a characteristic equation that
gives a modification of the lower bound \eqref{eq:thebound} on the root $\sigma^*$. Namely, in the diffusive case we have
\begin{align}
\sigma^{\ast} > \sigma^{(i)} > \frac{ a \mu m k_{2}^{2} (k_{1}^{2} + k_{2}^{2})}{2^{5} \Omega^{2} m^{2} (k_{1}^{2} + k_{2}^{2} + 4 m^{2}) + 2 \mu^{2} k_{2}^{4} } - \kappa (k_{1}^{2} + k_{2}^{2} + 4 m^{2}) \label{eq:diff:sigma:lower:bound}.
\end{align}
Therefore, for at most {\em finitely} many values of $k_{1}$ and $k_2$ we have $\sigma^{\ast} > 0$, and we may choose $a$ large enough in terms of $\kappa,m,\mu$ and $\Omega$, so that there exists at least one pair $(k_{1},k_2)$ giving rise to unstable eigenvalues.

More precisely, we observe that when $0 < \kappa \ll 1$, which is the range relevant to the Earth's fluid core, the maximum value of the lower bound given in \eqref{eq:diff:sigma:lower:bound} as a function of the wave numbers $\kkx$ and $\kky$ occurs when
\begin{align*}
\kkx \approx \frac{a}{2^{5} \Omega} \cdot \frac{1}{\kappa} \qquad \mbox{and}\qquad \kky \approx \frac{a^{1/2} m^{1/2}}{2\sqrt{2} \mu^{1/2}} \cdot \frac{1}{\kappa^{1/2}}
\end{align*}
for $0<\kappa \ll 1$. Inserting this into \eqref{eq:diff:sigma:lower:bound}, the lower bound on $\sigma^\ast$ becomes
\begin{align}
\sigma^\ast > \frac{a^2}{2^{10} \Omega^2} \cdot \frac{1}{\kappa} \label{eq:SigmaLowerBoundInTermsOfKappa}
\end{align}whenever $0<\kappa \ll 1$. Thus, the limit $\kappa \rightarrow 0$ is consistent with the non-diffusive problem studied in Section~\ref{sec:ill}, where  it was shown that there exist eigenvalues $\sigma$ that grow like $\kkx$ on the curves $\kkx \approx \kky^2$, as $\kkx \rightarrow \infty$.

The classical geodynamo problem of magnetic field generation via the flow of an electrically conducting fluid is closely connected with the existence of unstable eigenvalues (see, for example,~\cite{Moffatt2}). We have shown that the $\MGk$ equations are globally well-posed (cf.~\cite{FriedlanderVicol}) and that the linearized equations permit solutions that grow exponentially in time. We note that this growth in time for $\theta(\xx,t)$ implies the growth of the perturbation magnetic field ${\boldsymbol b}(\xx,t)$ via the relation \eqref{eq:magnetic}. The exponential growth rate is very rapid, i.e. of order $1/\kappa$, for very small buoyancy diffusivity $\kappa$.

\subsection*{Acknowledgements} The work of S.F. is supported in part by the NSF grant DMS 0803268. We are grateful to Benjamin Texier for pointing out that the definition of Lipschitz local well-posedness with $X=Y$ would be too restrictive in the non-diffusive setting.

\appendix
\section{On the higher regularity of the forced critical $\AMGk$ equations}~\label{appendix}

The goal of this section is to give a proof of Lemma~\ref{lemma:higher:regularity}. The main difficulty is to show that the constant $C^\ast$ in \eqref{eq:the:apriori} may be taken independently of time. Due to the triangle inequality
\begin{align}
\Vert \Delta \theta^\epsilon\Vert_{L^2} \leq \Vert \Delta \Theta^\epsilon\Vert_{L^2} + \Vert \Delta \Theta_0 \Vert_{L^2},\label{eq:ap:triangle}
\end{align}
in order to prove Lemma~\ref{lemma:higher:regularity} it is sufficient to consider bounds for the initial value problem associated to \eqref{eq:1}--\eqref{eq:2}, i.e.
\begin{align}
 &\partial_t \Teps + \Ueps \cdot \nabla \Teps - \kappa \Delta\Teps= S \label{eq:ap:1}\\
  & \nabla \cdot \Ueps= 0,\  U_{j}= \partial_i T_{ij} \Teps \label{eq:ap:2}\\
  & \Teps(\cdot,0) = \Th_0 + \epsilon \phi \label{eq:ap:3},
\end{align}where $S, \phi$, and $\Th_0$ are given, smooth, time independent functions, and $\epsilon \in (0,1)$.
It is clear that the following lemma, combined with \eqref{eq:ap:triangle} proves Lemma~\ref{lemma:higher:regularity}.
\begin{lemma}\label{lemma:ap}
Let $\Teps \in L^\infty_tL^2_x \cap L^2_t H^1_x$ be a mean-free global weak solution of \eqref{eq:ap:1}--\eqref{eq:ap:3}. There exists a positive constant $\bar{C} >0$ such that
\begin{align}\label{eq:ap:main}
\Vert \Delta \Teps(\cdot,t) \Vert_{L^2} \leq \bar{C}
\end{align}
for all $t\geq 0$, and all $\epsilon\in (0,1)$, where $\bar{C} = \bar{C}(\Vert S \Vert_{H^2}, \Vert \phi \Vert_{H^2}, \Vert \Th_0 \Vert_{H^2})$.
\end{lemma}

In \cite{FriedlanderVicol} we have proven that  weak solutions $\Th \in L^{\infty}_{t} L^{2}_{x} \cap L^{2}_{t} \dot{H}_{x}^{1}$ of the {\em unforced} $\AMGk$ equations \eqref{eq:1}--\eqref{eq:2}, evolving from merely $L^{2}$ initial data, are in fact H\"older continuous for positive time. Moreover, it is not hard to modify our argument (De Giorgi iteration scheme) in order to prove the same result if a source term $S \in L^{p}_{t,x,loc}$ is added, where $p > (d+2)/2$. Regarding the higher regularity of solutions, in \cite{FriedlanderVicol2} we have proven that weak solutions which are $C^\alpha$, with {\em any} positive H\"older exponent $\alpha$, are in fact $C^{\infty}$ smooth for positive time. Clearly if a $C^{\infty}$ source term is added to the equations the same result holds. Unfortunately, the bound given in \cite{FriedlanderVicol2} for the $H^2$ norm of the solution may grow with time. The arguments below show that in fact uniform in time upper bounds may be obtained. The only $\epsilon$-dependence of the below bounds is via $\Vert \Teps(\cdot,0) \Vert_{H^2}$. Since $\Vert \Teps(\cdot,0) \Vert_{H^2} \leq \Vert \Th_0 \Vert_{H^2} + \Vert \phi\Vert_{H^2}$, is $\epsilon$-independent, we shall henceforth ignore the $\epsilon$ dependence of all functions.

\begin{proof}[Proof of Lemma~\ref{lemma:ap}]
Since $\Th(\cdot,0) \in C^{\infty}$, by combining the local existence of smooth solutions with the instant regularization proven in \cite{FriedlanderVicol,FriedlanderVicol2}, and the uniqueness of smooth solutions, we \emph{a priori} obtain
\begin{align}
\sup_{0\leq t \leq T}\Vert \Delta \Th(\cdot,t) \Vert_{L^{2}} \leq C_{T} \label{eq:apriori:bound:2}
\end{align}for all positive time $T>0$, where $C_T < \infty \Leftrightarrow T<\infty$. We show that $C_T$ may be taken independent of $T$.

From the standard energy estimate and the Poincar\'e inequality ($\Th$ has zero mean), we obtain
\begin{align}
\frac{1}{2} \frac{d}{dt} \Vert \Th \Vert_{L^2}^{2} + \frac{\kappa}{C} \Vert \Th \Vert_{L^2}^{2} &\leq \frac{1}{2} \frac{d}{dt} \Vert \Th(\cdot,t)\Vert_{L^2}^{2} + \kappa \Vert \nabla \Th\Vert_{L^2}^{2}\leq \Vert S \Vert_{L^2} \Vert \Th \Vert_{L^2},\label{eq:ap:L2}
\end{align}
which shows that $\Vert \Th(\cdot,t)\Vert_{L^2}$ is uniformly bounded in time. Moreover, since $S\in L^\infty_{t,x}$, the first part of the De Giorgi iteration scheme (cf.~\cite{CaffarelliVasseur,FriedlanderVicol}) proves that
\begin{align}
\Vert \Theta(\cdot,t) \Vert_{L^\infty} \leq C \Vert \Theta (\cdot,0) \Vert_{L^2} (1+ t^{-3/4}) +C  \Vert S \Vert_{L^\infty},\label{eq:ap:Linfty}
\end{align}
for some universal constant $C>0$.  The above estimate for $t \geq 1$, combined with $T=1$  in \eqref{eq:apriori:bound:2}, shows that $\Vert \Th(\cdot,t)\Vert_{L^\infty}$ is uniformly bounded in time. Estimate \eqref{eq:ap:Linfty} may then be used in the second part of the De Giorgi iteration scheme to show that the $C^\alpha$-norm of $\Theta(\cdot,t)$ is also uniformly bounded in time. The latter holds since the H\"older norm of $\Theta(\cdot,t)$ in a parabolic cylinder $B_r \times [t_0-r^2,t_0]$, only depends on the $L^2$ and $L^\infty$ norms of $\Theta$ and $S$ on the larger (but finite)  parabolic cylinder $B_{2r} \times [t_0-3r^2,t_0+r^2]$.

We have so far shown that $\Vert \Th(\cdot,t)\Vert_{C^\alpha}$ is uniformly bounded in time, for some $\alpha>0$. The boundedness of higher order norms is now obtained via energy estimates. Multiplying \eqref{eq:ap:1} by $\Delta \Theta$ and integrating by parts, we obtain
\begin{align}
\frac 12 \frac{d}{dt} \Vert \nabla \Theta \Vert_{L^2}^2 + \kappa \Vert \Delta \Theta\Vert_{L^2}^2 &\leq \Vert S\Vert_{L^2} \Vert \Delta \Theta\Vert_{L^2} + \Vert \nabla \Ub \Vert_{L^2} \Vert \nabla \Theta \Vert_{L^4}^2.\label{eq:ap:ODE:1}
\end{align}To bound the $W^{1,4}$-norm of $\Th$, we interpolate between $C^\alpha$ and $H^2$. Using the Besov space characterization of the H\"older spaces $C^\alpha = B^{\alpha}_{\infty,\infty}$, we claim that the interpolation inequality
\begin{align}
\Vert \nabla \Theta \Vert_{L^4} \leq C \Vert \Theta \Vert_{L^2}^{5/8}  \Vert \Delta \Theta \Vert_{L^2}^{3/8} + C \Vert \Theta\Vert_{L^2}^{\alpha/8} \Vert \Delta \Theta \Vert_{L^2}^{1/2 - \alpha/8} \Vert \Theta \Vert_{C^\alpha}^{1/2}\label{eq:ap:interpolate}
\end{align}holds, for any $\alpha >0$ and sufficiently smooth $\Theta$. Indeed, since we work in three dimensions we have
\begin{align}
\Vert \nabla \Th \Vert_{L^4} &\leq \Vert S_0 \nabla \Th\Vert_{L^4} + \sum_{j\geq 0} \Vert \Delta_j \nabla \Th \Vert_{L^4}\notag\\
&\leq C \Vert \Th \Vert_{L^4} + C \sum_{j\geq 0} 2^j \Vert \Delta_j \Th \Vert_{L^2}^{1/2} \Vert \Delta_j  \Th \Vert_{L^\infty}^{1/2}\notag\\
&\leq C \Vert \Th \Vert_{L^2}^{5/8} \Vert \Delta \Th\Vert_{L^2}^{3/8} \notag\\
& \qquad +C \sum_{j\geq 0} \Vert \Delta_j \Th \Vert_{L^2}^{\alpha/8} \left( 2^{2j} \Vert \Delta_j \Th \Vert_{L^2} \right)^{1/2 - \alpha/8} \left( 2^{\alpha j} \Vert \Delta_j \Th \Vert_{L^\infty}\right)^{1/2} 2^{-j\alpha/4} \label{eq:Besov:estimate}
\end{align}
where $\Delta_j$ are standard Littlewood-Paley projection operators (cf.~\cite{ConstantinWu08,Triebel} and references therein).
Upon applying the discrete H\"older inequality, and using that $\alpha>0$, the bound \eqref{eq:Besov:estimate} implies estimate \eqref{eq:ap:interpolate}. While the above proof of \eqref{eq:ap:interpolate} applies in the case of the domain being ${\mathbb R}^3$, in order to prove \eqref{eq:ap:interpolate} for ${\mathbb T}^3$, one may use the equivalent Littlewood-Paley description of periodic Besov spaces cf.~\cite[Chapter 9]{Triebel} and an estimate similar to \eqref{eq:Besov:estimate}. Since this requires no new ideas, but is notationally heavy, we omit further details for the periodic case.

From \eqref{eq:ap:ODE:1} and \eqref{eq:ap:interpolate} we obtain the energy estimate
\begin{align}
\frac 12 \frac{d}{dt} \Vert \nabla \Theta \Vert_{L^2}^2 + \kappa \Vert \Delta\Theta\Vert_{L^2}^2 &\leq
 \Vert S\Vert_{L^2} \Vert \Delta \Theta\Vert_{L^2} + C \Vert \Theta \Vert_{L^2}^{5/4}  \Vert \Delta \Theta \Vert_{L^2}^{7/4}\notag\\
  & \qquad +\, C  \Vert \Theta \Vert_{L^2}^{\alpha/4}  \Vert \Delta \Theta \Vert_{L^2}^{2-\alpha/4}   \Vert \Theta \Vert_{C^\alpha}, \label{eq:ODE:interpolate}
\end{align}
for any $\alpha >0$.  It follows from \eqref{eq:ODE:interpolate}, the Cauchy-Schwartz inequality, and the Poincar\'e inequality (since $\Theta$ has zero mean) that
\begin{align*}
\frac 12 \frac{d}{dt} \Vert \nabla \Theta \Vert_{L^2}^2 + \frac{\kappa}{C} \Vert \nabla \Theta\Vert_{L^2}^{2} &\leq C_\kappa \Vert S\Vert_{L^2}^2 + C_\kappa \Vert \Theta \Vert_{L^2}^{10} + C_\kappa  \Vert  \Theta \Vert_{L^2}^2 \Vert \Theta\Vert_{C^\alpha}^{8/\alpha}.
\end{align*}
The uniform boundedness in time of $\Vert \Theta(\cdot,t) \Vert_{L^2}$ and of $\Vert \Theta(\cdot,t) \Vert_{C^\alpha}$ combined with the Gr\"onwall inequality proves that $ \Vert \nabla \Theta(\cdot,t) \Vert_{L^2}$ is uniformly bounded in time. A similar energy argument may be carried out by multiplying the evolution \eqref{eq:1} with $\Delta^2 \Theta$, in order to obtain that $ \Vert \Delta \Theta(\cdot,t) \Vert_{L^2}$ is uniformly bounded in time, concluding the proof of the lemma.
\end{proof}

\end{document}